\documentclass[a4paper,twopage,reqno,12pt]{amsart}
\usepackage[top=30mm,right=30mm,bottom=30mm,left=30mm]{geometry}

\usepackage{tabularx}
\usepackage{graphicx}
\usepackage{amsmath}
\usepackage{amssymb}
\usepackage{amsfonts}
\usepackage{amsthm}
\usepackage{amstext}
\usepackage{amsbsy}
\usepackage{amsopn}
\usepackage{amscd} 
\usepackage{enumerate}
\usepackage{color}
\usepackage{multirow}
\usepackage{longtable}
\usepackage{float}
\usepackage{lscape}
\usepackage{pdflscape}
\usepackage{url}

\newtheorem{theorem}{Theorem}[section]
\newtheorem{corollary}[theorem]{Corollary}
\newtheorem{lemma}[theorem]{Lemma}
\newtheorem{proposition}[theorem]{Proposition}

\theoremstyle{definition}

\newtheorem{remark}[theorem]{Remark}

\numberwithin{equation}{section}

\renewcommand{\leq}{\leqslant}
\renewcommand{\geq}{\geqslant}

\newcommand{\qbin}[2]{\begin{bmatrix}{#1}\\{#2}\end{bmatrix}_q}

\begin{document}
\title[Flag-transitive symmetric $2$-designs]{Flag-transitive, point-imprimitive symmetric $2$-$(v,k,\lambda )$ designs with $k>\lambda \left(\lambda-3 \right)/2$}

\author[]{ Alessandro Montinaro}

%\thanks{Corresponding author: Alessandro Montinaro}
%
%

\address{Alessandro Montinaro, Dipartimento di Matematica e Fisica “E. De Giorgi”, University of Salento, Lecce, Italy}
\email{alessandro.montinaro@unisalento.it}

\subjclass[MSC 2020:]{05B05; 05B25; 20B25}%
\date{\today}%
\dedicatory{Dedicated to Mauro Biliotti, professor emeritus at the University of Salento}
%\commby{}%

\begin{abstract}
Let $\mathcal{D}=\left(\mathcal{P},\mathcal{B} \right)$ be a symmetric $2$-$(v,k,\lambda )$ design admitting a flag-transitive, point-imprimitive automorphism group $G$ that leaves invariant a non-trivial partition $\Sigma$ of $\mathcal{P}$. Praeger and Zhou \cite{PZ} have shown that, there is a constant $k_{0}$ such that, for each $B \in \mathcal{B}$ and $\Delta \in \Sigma$, the size of $\left\vert B \cap \Delta \right \vert$ is either $0$ or $k_{0}$. In the present paper we show that, if $k>\lambda \left(\lambda-3 \right)/2$ and $k_{0} \geq 3$, $\mathcal{D}$ is isomorphic to one of the known flag-transitive, point-imprimitive symmetric $2$-designs with parameters $(45,12,3)$ or $(96,20,4)$.

%\bigskip

%\medskip

%\textbf{Keywords:} $2$-design, flag-transitive group.

%\textbf{MSC 2020:} 05B05, 05B25, 20B25.
\end{abstract}

\maketitle

\bigskip
\section{Introduction and Main Result}
A $2$-$(v,k,\lambda )$ \emph{design} $\mathcal{D}$ is a pair $(\mathcal{P},%
\mathcal{B})$ with a set $\mathcal{P}$ of $v$ points and a set $\mathcal{B}$
of blocks such that each block is a $k$-subset of $\mathcal{P}$ and each two
distinct points are contained in $\lambda $ blocks. We say $\mathcal{D}$ is 
\emph{non-trivial} if $2<k<v$, and symmetric if $v=b$. All $2$-$(v,k,\lambda )$ designs in this paper
are assumed to be non-trivial. An automorphism of $\mathcal{D}$ is a
permutation of the point set which preserves the block set. The set of all
automorphisms of $\mathcal{D}$ with the composition of permutations forms a
group, denoted by $\mathrm{Aut(\mathcal{D})}$. For a subgroup $G$ of $%
\mathrm{Aut(\mathcal{D})}$, $G$ is said to be \emph{point-primitive} if $G$
acts primitively on $\mathcal{P}$, and said to be \emph{point-imprimitive}
otherwise. In this setting we also say that $\mathcal{D}$ is either \emph{point-primitive} or \emph{point-imprimitive} respectively. A \emph{flag} of $\mathcal{D}$ is a pair $(x,B)$ where $x$ is a point and $B$
is a block containing $x$. If $G\leq \mathrm{Aut(\mathcal{D})}$ acts
transitively on the set of flags of $\mathcal{D}$, then we say that $G$ is 
\emph{flag-transitive} and that $\mathcal{D}$ is a \emph{flag-transitive
design}.\

Flag-transitive symmetric designs are widely studied. If $\lambda =1$, that is, $\mathcal{D}$ is a projective plane of order $n$, Kantor \cite{Ka87} proved that either $\mathcal{D}$ is Desarguesian and $PSL_{3}(n) \unlhd G$, or $G$ is a sharply flag-transitive Frobenius group of order $(n^{2} + n + 1)(n + 1)$, and $n^{2} + n + 1$ is a prime. In both cases the action of G is point-primitive. For $\lambda > 1$, flag-transitive, point-imprimitive symmetric designs do exist. In 1945 Hussain \cite{H} and, independently, in 1946 Nandi \cite{N} discovered that there are exactly three symmetric 2-(16, 6, 2)-designs. In 2006 O’Reilly Regueiro \cite{ORR} showed that, if $\lambda \leq 4$ then the parameters of $\mathcal{D}$ are $(16,2,2),(45,12,3),(15,8,4),(96,20,4)$ and that exactly two of the three $2$-designs discovered by Hussain and Nandi are flag-transitive and point-imprimitive. In 2006 Praeger and Zhou \cite{PZ} proved that there is exactly one flag-transitive, point-imprimitive symmetric $2$-$(15, 8,3)$ design, in 2007 Praeger \cite{P} showed that there is exactly one flag-transitive, point-imprimitive symmetric $2$-$(45, 12, 3)$ design, and in 2009, Law, Praeger and Reichard \cite{LPR} proved that there are exactly four transitive, point-imprimitive symmetric $2$-$(96, 40, 4)$ designs. Apart from two possible numerical exceptions, the classification of the flag-transitive point-imprimitive symmetric $2$-designs has been recently extended to $\lambda \leq 10$ by Mandi\'{c} and \v{S}ubasi\'{c} \cite{MS}. 

It is worth noting that one of the four $2$-$(96, 40, 4)$ designs is a special case of a beautiful, general construction of transitive, point-imprimitive symmetric $2$-designs due to Cameron and Praeger \cite{CP} based on a previous work of Sane \cite{S}. It is an open problem whether the remaining three $2$-designs arise or not from the Cameron-Praeger construction.\\
An upper bound on $k$, when $\mathcal{D}$ is flag-transitive and point-imprimitive, was given by O’Reilly Regueiro in \cite{ORR} and subsequently refined by Praeger and Zhou in \cite{PZ}. Among the other results, the authors determined the parameters of $\mathcal{D}$ as functions of $\lambda$ when $k>\lambda \left(\lambda -3 \right)/2$. Recently, the flag-transitive $2$-designs with $%
\lambda =2$ have been investigated by Devillers, Liang, Praeger and Xia in 
\cite{DLPX}, where, it is shown that, apart from the two known symmetric $2$-$%
(16,6,2)$ designs, $G$ is primitive of affine or almost simple type. \\
The present paper is a contribution to the problem of classifying flag-transitive, point-imprimitive symmetric $2$-$(v,k,\lambda)$ designs. We classify those with $k>\lambda \left(\lambda -3 \right)/2$ and such that a block of the $2$-design intersects a block of imprimitivity in at least $3$ points. More precisely, our result is the following.
\bigskip 
\begin{theorem}
\label{main} Let $\mathcal{D}=\left(\mathcal{P},\mathcal{B} \right)$ be a symmetric $2$-$(v,k,\lambda )$ design admitting a flag-transitive, point-imprimitive automorphism group $G$ that leaves invariant a non-trivial partition $\Sigma$ of $\mathcal{P}$. If $k> \lambda (\lambda-3)/2$ and there is block of $\mathcal{D}$ intersecting an element of $\Sigma$ in at least $3$ points, then one of the following holds:

\begin{enumerate}

\item $\mathcal{D}$ is isomorphic to the $2$-$(45,12,3)$ design of \cite[Construction 4.2]{P}.

\item $\mathcal{D}$ is isomorphic to one of the four $2$-$(96,20,4)$ designs constructed in \cite{LPR}.
\end{enumerate}
\end{theorem}

\bigskip

The outline of the proof is as follows. The group $G$ preserves a set of imprimitivity $\Sigma$ on the point set of $\mathcal{D}$ consisting of $d$ blocks of imprimitivity each of size $c$. By \cite{PZ} each block $B$ of $\mathcal{D}$ intersects any block of imprimitivity either in $0$ or in a constant number $k_{0}$ of points. In Lemma \ref{Rid1} we show that the number of blocks intersecting a block of imprimitivity in the same $k_{0}$-set of points is constant and is independent on the choice of the block of $\mathcal{D}$ and of the element of $\Sigma$. We call such a number \emph{the overlap number of $\mathcal{D}$} and we denote it by $\theta$. \\ If $k_{0} \geq 3$, in Theorems \ref{des} and \ref{PZM} we show that the blocks of imprimitivity have the structure of flag-transitive $2$-$(c,k_{0},\lambda/\theta)$ designs, where $(c,k_{0})$ is either $(\lambda^{2},\lambda)$, or $(\lambda +6, 3)$ with $\lambda \equiv 1,3 \pmod{6}$. Moreover, in Lemma \ref{prim} we prove that, such $2$-designs are also point-primitive. Flag-transitive, point-primitive $2$-$(\lambda^{2},\lambda,\lambda/\theta)$ designs are classified  in \cite{MF,Mo1,Mo2}, whereas flag-transitive, point-primitive $2$-$(\lambda +6 ,3,\lambda/\theta)$ designs, with $\lambda \equiv 1,3 \pmod{6}$ are classified in the Appendix of this paper (Theorem \ref{tre}). Finally, we complete the proof of Theorem \ref{main} by combining the previous information on the structure of the blocks of imprimitivity with the constraints on the action of $G$ on $\mathcal{D}$ provided in \cite{La} and on the structure of $G$ essentially provided in \cite{BP}.

\section{The overlap number of $\mathcal{D}$}

Let $\mathcal{D}=\left(\mathcal{P},\mathcal{B} \right)$ be a symmetric $2$-$(v,k,\lambda )$ design admitting a flag-transitive, point-imprimitive automorphism group $G$ that leaves invariant a non-trivial partition $\Sigma$ of $\mathcal{P}$ with $d$ classes of size $c$. Then there is a constant $k_{0}$ such that, for each $B \in \mathcal{B}$ and $\Delta_{i} \in \Sigma$, $i=1,...,d$, the size $\left\vert B \cap \Delta_{i} \right \vert$ is either $0$ or $k_{0}$ by \cite[Theorem 1.1]{PZ}. If we pick two distinct points $x,y$ in a block of imprimitivity, then there are exactly $\lambda $ blocks of $\mathcal{D}$ incident with them. Thus $k_{0} \geq 2$. Moreover, since $\mathcal{D}$ is  non-trivial, $v>k$ and hence $k_{0}<c$ by \cite[(4) and (7)]{PZ}. Therefore, $2 \leq k_{0}<c$. If $k_{0}=2$ the flag-transitivity of $G$ on $\mathcal{D}$ implies the $2$-transitivity of $G_{\Delta _{i}}^{\Delta _{i}}$ on $\Delta_{i}$ for each $i=1,...,d$. 

\bigskip

Let $\mathcal{B}_{i}=\left\{ B\in \mathcal{B}:B\cap \Delta _{i}\neq
\varnothing \right\} $, where $i=1,...,d$. For any $B\in \mathcal{B}_{i}$ define
\bigskip
\begin{equation*}
\mathcal{B}_{i}(B)=\left\{ B^{\prime }\in \mathcal{B}_{i}:B^{\prime }\cap
\Delta _{i}=B\cap \Delta _{i}\right\} \; \text{ and } \; \theta (i,B)=\left\vert 
\mathcal{B}_{i}(B)\right\vert .
\end{equation*}
\text{} \\
\bigskip
Clearly, $ 1 \leq \theta (i,B)\leq \lambda $.

\begin{lemma}
\label{Rid1}$\theta (i,B)=\theta (j,B^{\prime })$ for each $i,j\in \left\{
1,...,d\right\} $ and for each $B\in \mathcal{B}_{i}$ and $B^{\prime }\in 
\mathcal{B}_{j}$.
\end{lemma}

\begin{proof}
Let $B\in \mathcal{B}_{i}$ and $B^{\prime }\in \mathcal{B}_{j}$, where $i,j\in \left\{ 1,...,d\right\} $, and let $%
x\in B\cap \Delta _{i}$ and $x^{\prime }\in B^{\prime }\cap \Delta _{j}$. Then there is $\psi \in G$ such
that $(x,B)^{\psi }=(x^{\prime },B^{\prime })$ since $G$ is flag-transitive. Hence $\left( B\cap \Delta _{i}\right) ^{\gamma }=B^{\prime }\cap \Delta _{j}$. \\ Let $C\in \mathcal{B}_{i}(B)$, then $C\cap
\Delta _{i}=B\cap \Delta _{i}$ and hence 
\begin{equation*}
C^{\gamma }\cap \Delta _{j}=\left( C\cap \Delta _{i}\right) ^{\gamma
}=\left( B\cap \Delta _{i}\right) ^{\gamma }=B^{\prime }\cap \Delta _{j}\text{.}
\end{equation*}%
Thus $C^{\gamma }\in \mathcal{B}_{j}(B^{\prime })$ and hence $\theta
(i,B)\leq \theta (j,B^{\prime })$. Now, switching the role of $B$ and $%
B^{\prime }$ in the previous argument we get $\theta (j,B^{\prime })\leq
\theta (i,B)$. Thus $\theta (i,B)=\theta (j,B^{\prime })$, which is the assertion.
\end{proof}

\bigskip

In view of the previous lemma, we may denote $\theta (i,B)$ simply by $%
\theta $ and call it \emph{the overlap number of $\mathcal{D}$}.

\bigskip

\begin{corollary}
\label{ermo}Let $B\in \mathcal{B}_{i}$ and let $x\in B\cap \Delta _{i}$, then $\theta =\left[ G_{x,B\cap \Delta _{i}}:G_{x,B}\right] $.
\end{corollary}

\begin{proof}
Let $B\in \mathcal{B}_{i}$ and let $x\in B\cap \Delta _{i}$. Then $%
G_{x,B}\leq G_{x,B\cap \Delta _{i}}$. Thus $\left[ G_{x,B\cap \Delta
_{i}}:G_{x,B}\right] \leq \theta $.

Let $B^{\prime }\in \mathcal{B}_{i}(B)$. Then there is $\varphi \in G_{x}$
such that $B^{\varphi }=B^{\prime }$. Thus $(B\cap \Delta _{i})^{\varphi
}=B^{\prime }\cap \Delta _{i}=B\cap \Delta _{i}$ and hence $\varphi \in
G_{x,B\cap \Delta _{i}}$ and $G_{x,B}\varphi \subseteq G_{x,B\cap \Delta _{i}}$. Therefore $\left[ G_{x,B\cap \Delta _{i}}:G_{x,B}\right]
\geq \theta $ and hence $\left[ G_{x,B\cap \Delta _{i}}:G_{x,B}\right]
=\theta $.
\end{proof}

\begin{theorem}\label{des}
If $k_{0} \geq 3$, then $\mathcal{D}_{i}=(\Delta _{i},\mathcal{B}_{i}^{\ast})$, where $\mathcal{B}_{i}^{\ast}=\{B\cap \Delta _{i}:B \in \mathcal{B}_{i} \}$, is a non-trivial $2$-$%
\left( c,k_{0},\lambda /\theta \right) $ design, with $\theta \mid \lambda $,
admitting $G_{\Delta _{i}}^{\Delta _{i}}$ as a flag-transitive automorphism
group.
\end{theorem}

\begin{proof}
Clearly the number of points in $\mathcal{D}_{i}$ is $c$ and each
element of $\mathcal{B}_{i}^{\ast}$ contains $k_{0}$ points of $\Delta _{i}$. Let $x_{1},x_{2}\in \Delta _{i}$, with $x_{1}\neq
x_{2}$, then there are precisely $\lambda$ blocks of $%
\mathcal{D}$ incident with them, say $B_{1},...,B_{\lambda }$. For each $B_{j}$ there are
precisely $\theta $ blocks among $B_{1},...,B_{\lambda }$ whose intersection set with $\Delta_{i}$ is $B_{j} \cap \Delta_{i}$, hence there are exactly $\lambda /\theta $ distinct elements of $\mathcal{B}_{i}^{\ast}$ incident
with $x_{1}\neq x_{2}$. Thus $\mathcal{D}_{i}$ is a $2$-$\left(
c,k_{0},\lambda /\theta \right) $ design. Also, $\mathcal{D}_{i}$ is non-trivial since $k_{0}<c$ by \cite[(4) and (7)]{PZ}, and since $k_{0} \geq 3$ by our assumption. Finally, the flag-transitivity of 
$G$ on $\mathcal{D}$ implies the flag-transitivity of $G_{\Delta
_{i}}^{\Delta _{i}}$ on $\mathcal{D}_{i}$.
\end{proof}

\bigskip

The following theorem is an improvement of \cite[Theorem 1.1]{PZ} on the basis of Theorem \ref{des}.

\bigskip

\begin{theorem}
\label{PZM}Let $\mathcal{D}=(\mathcal{P} ,\mathcal{B})$ be a symmetric $2$-design admitting a flag-transitive, point-imprimitive automorphism group $G$ that leaves invariant a non-trivial partition $\Sigma =\left\lbrace \Delta_{1},...,\Delta_{d} \right\rbrace$
of $\mathcal{P} $ such that $\left\vert \Delta_{i} \right \vert =c$ for each $i=1,...,d$. Then the following hold:
\begin{enumerate}
\item[I.] There is a constant $k_{0}$ such that, for each $B\in \mathcal{B}$ and $\Delta _{i}\in \Sigma $, the size $\left\vert B\cap \Delta _{i}\right\vert $ is either $0$ or $k_{0}$.
\item[II.] There is a constant $\theta$ such that, for each $B\in \mathcal{B}$ and $\Delta _{i}\in \Sigma $ with $\left\vert B\cap \Delta _{i}\right\vert >0$, the number of blocks of $\mathcal{D}$ whose intersection set with $\Delta _{i}$ coincides with $B\cap \Delta _{i}$ is $\theta$.
\item[III.] If $k_{0}=2$ then $G_{\Delta _{i}}^{\Delta _{i}}$ acts $2$-transitively on $\Delta_{i}$ for each $i=1,...,d$.
\item[IV.] If $k_{0}\geq 3$ then $\mathcal{D}_{i}=\left( \Delta_{i}, \left(B \cap \Delta_{i} \right)^{G_{\Delta _{i}}^{\Delta _{i}}}\right)$ is a flag-transitive non-trivial $2$-$\left(c,k_{0},\lambda/\theta\right)$ design for each $i=1,...,d$.
\end{enumerate}
Moreover, if $k>\lambda (\lambda -3)/2$ then one of the following holds: 
\begin{enumerate}
\item[V.] $k_{0}=2$ and one of the following holds:
\begin{enumerate}
\item[1.] $\mathcal{D}$ is a symmetric $2$-$(\lambda ^{2}(\lambda +2),\lambda (\lambda
+1),\lambda )$ design and $\left(c,d \right)=\left(\lambda+2,\lambda^{2}\right)$.
\item[2.] $\mathcal{D}$ is a symmetric $2$-$\left( \left(\frac{\lambda+2}{2}\right) \left(\frac{\lambda^2-2\lambda+2}{2}\right),\frac{\lambda^2}{2},\lambda \right)$ design, $\left(c,d \right)=\left(\frac{\lambda+2}{2},\frac{\lambda^2-2\lambda+2}{2}\right)$, and either $\lambda \equiv 0 \pmod{4}$, or $\lambda=2u^{2}$, where $u$ is odd, $u \geq 3$, and $2(u^2-1)$ is a square. 
\end{enumerate}
\item[VI.] $k_{0}\geq 3$ and one of the following holds:
\begin{enumerate}
\item[1.] $\mathcal{D}$ is a symmetric $2$-$(\lambda ^{2}(\lambda +2),\lambda (\lambda
+1),\lambda )$ design, $d=\lambda +2$, and $\mathcal{D}_{i}$ is a $%
2 $-$(\lambda ^{2},\lambda ,\lambda /\theta )$ design, with $\theta \mid
\lambda $, for each $i=1,...,d$.

\item[2.] $\mathcal{D}$ is a symmetric $2$-$((\lambda +6)\frac{\lambda ^{2}+4\lambda -1}{4}%
,\lambda \frac{\lambda +5}{2},\lambda )$ design, with $\lambda
\equiv 1,3\pmod{6}$, $d=\frac{\lambda ^{2}+4\lambda -1}{4}$, and $\mathcal{%
D}_{i}$ is a $2$-$(\lambda +6,3,\lambda /\theta )$ design, with $\theta \mid
\lambda $, for each $i=1,...,d$.
\end{enumerate}
\end{enumerate}
\end{theorem}

\begin{proof}
The assertion follows from \cite[Theorem 1.1]{PZ} and from Theorem \ref{des}. 
\end{proof}

\bigskip
\textbf{From now on we assume that $k>\lambda (\lambda -3)/2$ and $k_{0}\geq 3$.} Hence, we will focus on the symmetric $2$-designs in (VI.1) and (VI.2) of Theorem \ref{PZM}, and we will refer to them as \emph{$2$-designs of type 1 and 2} respectively.
\bigskip

\begin{lemma}
\label{c=d} $\lambda \geq 3$.
\end{lemma}

\begin{proof}
If $\lambda=2$, then $k_{0}=2$ by \cite[Corollary 1.3 and Table 1]{PZ}, which is contrary to our assumption. Thus $\lambda \geq 3$.
\end{proof}

\bigskip

\begin{lemma}
\label{prim}If $k>\lambda (\lambda -3)/2$ and $k_{0}\geq 3$, then $G_{\Delta _{i}}^{\Delta _{i}}$ acts point-primitively on $\mathcal{D}_{i}$.

\end{lemma}

\begin{proof}
The assertion follows from \cite[2.3.7.(c)]{Demb} or \cite[Theorem 4.8.(i)]{Ka69}.
\end{proof}

\bigskip 
\begin{lemma}
\label{Ordine}Let $N$ be a minimal normal subgroup of $G$. Then one of the
following holds:

\begin{enumerate}
\item $\Sigma $ is the $N$-orbit decomposition of the point set of $\mathcal{D}$;
\item $N$ acts point-transitively on $\mathcal{D}$;
\end{enumerate}
or, for $c=\lambda+2$ and $d=\lambda^2$ the following additional possibility arises:
\begin{enumerate} 
\item[(3)] The $N$-orbit decomposition of the point set of $\mathcal{D}$ is a further $G$-invariant partition $\Sigma^{\prime}=\left\lbrace \Delta_{1}^{\prime},...,\Delta_{\lambda^{2}}^{\prime} \right\rbrace$ such that the following hold:
\begin{enumerate}
\item[(a)] $\left\vert \Delta_{j}^{\prime} \right \vert =\lambda+2$ for each $j=1,...,\lambda^{2}$;
\item[(b)] For each $B\in \mathcal{B}$ and $\Delta _{j}^{\prime}\in \Sigma^{\prime} $, the size $\left\vert B\cap \Delta _{j}^{\prime}\right\vert $ is either $0$ or $2$;
\item[(c)] For each $\Delta _{i}\in \Sigma$ and $\Delta _{j}^{\prime}\in \Sigma^{\prime} $, $\left\vert \Delta _{i}\cap \Delta _{j}^{\prime}\right\vert =1$.
\item[(d)] $G_{\Delta _{j}^{\prime}}^{\Delta _{j}^{\prime}}$ acts $2$-transitively on $\Delta_{j}^{\prime}$ for each $j=1,...,\lambda^{2}$.
\end{enumerate}   
\end{enumerate}
\end{lemma}

\begin{proof}
Let $N$ be a minimal normal subgroup of $G$. 
Assume that $G_{\Delta _{i}}N$ acts point-transitively on $\mathcal{D}$.
Then $N$ acts transitively on $\Sigma $. If there is $j_{0}\in \left\{ 1,...,d\right\} $ such that $N_{\Delta_{j_{0}}}^{\Delta_{j_{0}}}=1$, then $N_{\Delta
_{j_{0}}}\leq G(\Delta _{j_{0}})$. Hence $N_{\Delta _{i}}\leq G(\Delta _{i})$
 for each $i$, since $G$ acts transitively on $\Sigma $ and $N \trianglelefteq G$. Thus the point set of $\mathcal{D}$ is split into $%
c^{\prime }$ orbits under $N$ each of length $d^{\prime }
$, where $(c^{\prime },d^{\prime })=(d,c)$, since $N$ acts transitively on $\Sigma $. Hence $\Sigma
^{\prime }=\left\{ \Delta _{1}^{\prime },...,\Delta
_{c}^{\prime }\right\} $, where $\Delta _{j}^{\prime 
}=x_{j}^{N}$ for each $j=1,...,c$, is a set of imprimitivity for $G$. Moreover, $N_{x_{i}}=N(\Delta_{i})$ for each $x_{i} \in \Delta_{i}$ and for each $i=1,...,c$. By Theorem \ref{PZM} there is a constant $k_{0}^{\prime }$ such that, for each $B\in \mathcal{B}$ and $\Delta _{i}^{\prime }\in \Sigma^{\prime } $, the size $\left\vert B\cap \Delta _{i}^{\prime }\right\vert $ is either $0$ or $k_{0}^{\prime }$.\\
If $k_{0}^{\prime }=2$, then either $\left(c^{\prime },d^{\prime } \right)=\left(\lambda+2,\lambda^{2}\right)$, or $\left(c^{\prime  },d^{\prime  } \right)=\left(\frac{\lambda+2}{2},\frac{\lambda^2-2\lambda+2}{2}\right)$ and either $\lambda \equiv 0 \pmod{4}$, or $\lambda=2u^{2}$, where $u$ is odd, $u \geq 3$, and $2(u^2-1)$ is a square by Theorem \ref{PZM}. On the other hand, we know that $(c^{ \prime },d^{\prime })=(d,c)$ and either $(d,c)=(\lambda+2, \lambda^2)$, or $\left( \frac{\lambda ^{2}+4\lambda -1}{4}, \lambda+6 \right)$ and $\lambda \equiv 1,3 \pmod{6}$ again by by Theorem \ref{PZM}, since $k_{0} \geq 3$. By comparing the values of $(c^{\prime },d^{\prime })$ we see that
the unique admissible value is $(c^{\prime },d^{\prime })=(d,c)=(\lambda+2, \lambda^2)$, and we obtain (3a) and (3b).\\
Let $\Delta _{i}\in \Sigma$ and $\Delta _{j}^{\prime}\in \Sigma^{\prime} $. Since $N_{x_{i}}=N(\Delta_{i})$ for each $x_{i} \in \Delta_{i}$ and for each $i=1,...,\lambda+2$, and since $\Delta _{j}^{\prime}$ is a $N$-orbit for each $j=1,...,\lambda^{2}$, it follows that $\left\vert \Delta _{i}\cap \Delta _{j}^{\prime}\right\vert =1$. Also $G_{\Delta _{j}^{\prime}}^{\Delta _{j}^{\prime}}$ acts $2$-transitively on $\Delta_{j}^{\prime} \in \Sigma^{\prime}$, since $k_{0}^{\prime}=2$. Thus we get (3c) and (3d).\\ 
If  $k_{0}^{\prime } \geq 3$, then either $\left(c^{\prime },d^{\prime } \right)=\left(\lambda^{2}, \lambda +2\right)$ or $\left(c^{\prime  },d^{\prime } \right)=\left(\lambda+6,\frac{\lambda ^{2}+4\lambda -1}{4} \right)$ and $\lambda \equiv 1,3 \pmod{6}$ by Theorem \ref{PZM}. On the other hand, we know that $(c^{\prime },d^{\prime  })=(d,c)$ and either $(d,c)=(\lambda+2, \lambda^2)$, or $\left( \frac{\lambda ^{2}+4\lambda -1}{4}, \lambda+6 \right)$ and $\lambda \equiv 1,3 \pmod{6}$. By comparing the values of $(c^{\prime },d^{\prime })$ no admissible $\lambda$'s arise, since $\lambda \ge 3$ by Lemma \ref{c=d}.

Assume that $%
N_{\Delta _{i}}^{\Delta _{i}} \neq 1$ for each $i=1,...,d$. Hence $%
N_{\Delta _{i}}^{\Delta _{i}}$ acts point-transitively on $\mathcal{D}_{i}$ for each $i= 1,...,d$, since $G_{\Delta
_{i}}^{\Delta _{i}}$ acts point-primitively on $\mathcal{D}_{i}$ by
Lemma \ref{prim}. Therefore $N$ acts point-transitively on $\mathcal{%
D}$, as $N$ acts transitively on $\Sigma $, which is (2).

Assume that $G_{\Delta _{i}}N$
acts point-intransitively on $\mathcal{D}$. Hence $G_{\Delta _{i}}N \neq G$. Then $\Delta _{i}\subseteq
\Delta _{i}^{\prime \prime}$, where $\Delta _{i}^{\prime \prime}=x^{G_{\Delta
_{i}}N}=\Delta _{i}^{N}$ and $x\in \Delta _{i}$. Also $\Sigma ^{\prime \prime
}=\left\{ \left( \Delta _{i}^{\prime \prime}\right) ^{g}:g\in G\right\} $ is a set
of imprimitivity for $G$ by \cite[Theorem 1.5A]{DM}. If $B\in \mathcal{B}$ is such that $B\cap \Delta
_{i}^{\prime \prime}\neq \varnothing $, then $k_{0}^{\prime \prime}=\left\vert B\cap \Delta _{i}^{\prime \prime
}\right\vert \geq \left\vert B\cap \Delta _{i}\right\vert=k_{0} \geq 3$ and hence
we may apply Theorem \ref{PZM} referred to the set of imprimitivity $\Sigma ^{\prime \prime}
$, and we obtain that $\mathcal{D}_{i}^{\prime \prime}=( \Delta _{i}^{\prime \prime
},\left(B \cap \Delta_{i}^{\prime \prime} \right)^{G_{\Delta _{i}^{\prime \prime}}^{\Delta _{i}^{\prime \prime}}}) $ is a flag-transitive non-trivial $2$-$(c^{\prime \prime},k_{0}^{\prime \prime},\lambda
/\theta ^{\prime \prime})$ design. Moreover, either $c^{\prime \prime}=\lambda ^{2}$ or $c^{\prime \prime
}=\lambda +6$ since $k>\lambda(\lambda-3)/2$.
It is easily seen that $c^{\prime \prime}=c$, since $c \mid c^{\prime \prime}$, being $\Delta _{i}^{\prime \prime}=\Delta _{i}^{N}$. Thus $\Delta _{i}=\Delta
_{i}^{\prime \prime}$ and hence $N \trianglelefteq G_{\Delta_{i}}$ for each $i=1,...,d$. If there $i_{0}$ such that $N$ fixes a point in $\Delta_{i_{0}}$, then $N$ fixes each point of $\mathcal{D}$, since $N \trianglelefteq G$ and $G$ acts point-transitively on $\mathcal{D}$, and we reach a contradiction. Thus $N^{\Delta_{i}} \neq 1$ for each $i=1,...,d$, and hence $N$ acts point-transitively on $\mathcal{D}_{i}$, since $N^{\Delta_{i}}\trianglelefteq G_{\Delta _{i}}^{\Delta _{i}}$ and since $G_{\Delta _{i}}^{\Delta _{i}}$ acts point-primitively on $\mathcal{D}_{i}$ by Lemma \ref{prim}. Therefore, $\Sigma $ is the orbit decomposition of the point set of $\mathcal{D}$ under $N$, which is (1).
\end{proof}

\bigskip

Let $\Delta \in \Sigma $ and $x \in \Delta $. Since $G(\Sigma) \trianglelefteq G_{\Delta}$ and $G(\Delta) \trianglelefteq G_{x}$ it is immediate to verify that $(G^{\Sigma})_{\Delta}=(G_{\Delta})^{\Sigma}$ and that $\left( G_{\Delta}^{\Delta} \right)_{x}=(G_{x})^{\Delta}$. Hence, in the sequel, $(G^{\Sigma})_{\Delta}$ and $\left( G_{\Delta}^{\Delta} \right)_{x}$ will simply be denoted $G^{\Sigma}_{\Delta}$ and $G_{x}^{\Delta}$ respectively.

\bigskip

\section{The case where $\mathcal{D}$ is of type 1}
In this section we assume that $\mathcal{D}$ is of type 1. Hence
$\mathcal{D}$ is a symmetric $2$-$(\lambda ^{2}(\lambda +2),\lambda (\lambda
+1),\lambda )$ design with $d=\lambda +2$. Moreover, $\mathcal{D}_{i}$ is a $%
2 $-$(\lambda ^{2},\lambda ,\lambda /\theta )$ design, with $\theta \mid
\lambda $, admitting $G_{\Delta _{i}}^{\Delta _{i}}$ as a flag-transitive, point-primitive
automorphism group for each $i=1,...,d$. Our aim is to prove the following result.

\bigskip

\begin{theorem}
\label{t1}If $\mathcal{D}$ is of type 1, one of the following holds:
\begin{enumerate}

\item $\mathcal{D}$ is isomorphic to the $2$-$(45,12,3)$ design of \cite[Construction 4.2]{P}.

\item $\mathcal{D}$ is isomorphic to one of the four $2$-$(96,20,4)$ designs constructed in \cite{LPR}.
\end{enumerate}
\end{theorem}

\bigskip

\begin{proposition}
\label{boom}$G$ induces a $2$-transitive group on $\Sigma $.
\end{proposition}

\begin{proof}
It is clear that $G$ acts transitively on $\Sigma $. Let $B$ be any block of $\mathcal{D}$ and define $\Sigma (B)=\left\{ \Delta_{i}\in \Sigma : \allowbreak  \Delta _{i}\cap B\neq \varnothing \right\}$. Then $%
\left\vert \Delta _{i}\cap B\right\vert =\lambda $ for each $\Delta _{i} \in \Sigma
(B)$, $\left\vert \Sigma
(B)\right\vert =\lambda +1$ and $\Sigma \setminus \Sigma (B)=\left\{ \Delta _{i_{0}}\right\} $ for some $i_{0} \in \left\lbrace 1,...,\lambda+2 \right\rbrace$, since $k=\lambda (\lambda +1)$ and $%
\left\vert \Sigma \right\vert =\lambda +2$. Since $G_{B}$ acts transitively
on $B$ and preserves $\Sigma $, it follows that $G_{B}$ acts transitively on 
$\Sigma (B)$. Thus $G_{B}$ preserves $\Delta _{i_{0}}$ and hence $G_{B}\leq
G_{\Delta _{i_{0}}}$. Therefore $G_{\Delta _{i_{0}}}$ acts transitively on $\Sigma \setminus \left\{
\Delta _{i_{0}}\right\} $ and hence $G$ induces a $2$-transitive group on $%
\Sigma $.
\end{proof}

\bigskip

\begin{lemma}
\label{NoFTR}If $G(\Delta _{i})\neq 1$, then either the primes dividing the order
of $G(\Delta _{i})$ divide $\lambda $, or $\mathcal{D}_{i}$ is a translation
plane.
\end{lemma}

\begin{proof}
Assume that $G(\Delta _{i})\neq 1$ and let $W$ be any Sylow $w$-subgroup of $%
G(\Delta _{i})$ where $w$ is a prime not dividing $\lambda $. Clearly $W$
fixes the $\frac{\lambda ^{2}}{\theta }(\lambda +1)$ blocks of $\mathcal{D}_{i}$%
. Let $B$ be any block of $\mathcal{D}$ such that $B \cap \Delta_{i}$ is a block of $\mathcal{D}_{i}$. Then $W$ preserves $B \cap \Delta_{i}$ and there are $\theta $ blocks of $\mathcal{D}$ whose intersection set with $\Delta _{i}$ is $B \cap \Delta_{i}$. Therefore $W$ fixes at least one of these $\theta $ blocks, as $w \nmid \lambda$ and $\theta \mid \lambda $, and hence $W$ fixes at least $\frac{%
\lambda ^{2}}{\theta }(\lambda +1)$ blocks of $\mathcal{D}$. Then any non-trivial element of $W$ fixes
at least $\frac{\lambda ^{2}}{\theta }(\lambda +1)$ points of $\mathcal{D}$
by \cite[Theorem 3.1]{La} and hence $\frac{\lambda ^{2}}{\theta }(\lambda
+1)\leq \frac{\lambda }{k-\sqrt{(k-\lambda )}}v$ by \cite[Corollary
3.7]{La}. Since $k=\lambda (\lambda +1)$ and $v=\lambda ^{2}(\lambda +2)$, it
follows that $\frac{\lambda ^{2}}{\theta }(\lambda +1)\leq \lambda (\lambda
+2)$. Thus $\theta =\lambda $ and hence $\mathcal{D}_{i}$ is a $2$-$(\lambda^2,\lambda,1)$ design, that is, an affine plane. Then $\mathcal{D}_{i}$ is a translation plane by \cite{Wa}, since $G_{\Delta _{i}}^{\Delta _{i}}$ acts flag-transitively on $\mathcal{D}_{i}$.
\end{proof}

\bigskip
The following theorem classifies the flag-transitive $2$-$(\lambda ^{2},\lambda ,\lambda
/\theta )$ designs $\mathcal{D}_{i}$.
\bigskip
\begin{theorem}
\label{Monty}If $\mathcal{D}_{i}$ is a $2$-$(\lambda ^{2},\lambda ,\lambda
/\theta )$ design admitting a flag-transitive automorphism group $G_{\Delta
_{i}}^{\Delta _{i}}$, then one of the following holds
\begin{enumerate}
\item $G_{\Delta _{i}}^{\Delta _{i}}$ is almost simple and one of
the following holds:

\begin{enumerate}
\item $\mathcal{D}_{i}$ is isomorphic to the $2$-$(6^{2},6,2)$ design
constructed in \cite{MF}, $\theta =3$ and $PSL_{2}(8)\trianglelefteq
G_{\Delta _{i}}^{\Delta _{i}}\leq P\Gamma L_{2}(8)$.

\item $\mathcal{D}_{i}$ is isomorphic to one of the three $2$-$(6^{2},6,6)$
designs constructed in \cite{MF}, $\theta =1$ and $G_{\Delta _{i}}^{\Delta
_{i}}\cong P\Gamma L_{2}(8)$.

\item $\mathcal{D}_{i}$ is isomorphic to the $2$-$(12^{2},12,3)$ design
constructed in \cite{Mo1}, $\theta =4$ and $G_{\Delta _{i}}^{\Delta
_{i}}\cong PSL_{3}(3)$.

\item $\mathcal{D}_{i}$ is isomorphic to the $2$-$(12^{2},12,6)$ design
constructed in \cite{Mo1}, $\theta =2$ and $G_{\Delta_{i}}^{\Delta _{i}}\cong PSL_{3}(3):Z_{2}$.
\end{enumerate}
\item  $G_{\Delta _{i}}^{\Delta _{i}}=T:G_{0}^{\Delta _{i}}$, $\lambda=p^{m}$, $p$ prime, $m \geq 1$, and one of the following holds:
\begin{enumerate}
\item $\mathcal{D}_{i}$ is a translation plane of order $p^{m}$, $\theta=p^{m}$, and one of the following holds:
\begin{enumerate}
\item $\mathcal{D}_{i} \cong AG_{2}(p^{m})$ and the possibilities $G_{0}^{\Delta _{i}}$ are given \cite{F1,LiebF}.
\item $\mathcal{D}_{i}$ is the L\"uneburg plane of order $2^{m}$, $m \equiv 2 \pmod{4}$, $m \geq 6$, and $Sz(2^{m/2}) \trianglelefteq G_{0}^{\Delta _{i}} \leq \left( Z_{2^{m/2}-1} \times Sz(2^{m/2}) \right).Z_{m/2}$;
\item $\mathcal{D}_{i}$ is the Hall plane of order $3^2$ and the possibilities for $G_{0}^{\Delta _{i}}$ are given \cite{F2};
\item $\mathcal{D}_{i}$ is the Hering plane of order $3^3$ and $G_{0}^{\Delta _{i}} \cong SL_{2}(13)$.
\end{enumerate}
\item $\mathcal{D}_{i}$ is a $2$-$(p^{2m},p^{m}, p^{m-t})$ design, $\theta=p^{t}$, where $0 \leq t \leq m$, the blocks are subspaces of $AG_{2m}(p)$ and $G_{0}^{\Delta _{i}} \leq \Gamma L_{1}(p^{2m})$;
\item $\mathcal{D}_{i}$ is isomorphic to one of the following $2$-designs constructed in \cite{Mo2}:
\begin{enumerate}
\item a $2$-$(p^{2m},p^{m},p^{m/2})$ design, $m$ even, $\theta=p^{m/2}$, and $SL_{2}(p^{m})\trianglelefteq G_{0}^{\Delta _{i}} \leq (Z_{p^{m/2}-1} \circ SL_{2}(p^{m})).Z_{m}$;
\item a $2$-$(p^{2m},p^{m},p^{2m/3})$ design, $p$
odd and $m \equiv 0 \pmod{3}$, $\theta=p^{m/3}$, and $SU_{3}(p^{m/3})\trianglelefteq G_{0}^{\Delta _{i}}\leq (Z_{p^{m/3}-1}\times
SU_{3}(p^{m/3})).Z_{2m/3}$.
\item $\mathcal{D}$ is a $2$-$(p^{2m},p^{m},p^{m})$ design, $m$ even, $\theta=1$, and $Sp_{4}(p^{m/2})%
\trianglelefteq G_{0}^{\Delta _{i}}\leq \Gamma Sp_{4}(p^{m/2})$.
\item a $2$-$(2^{2m},2^{m},2^{m}/\theta)$ design, $m \equiv 2 \pmod{4}$ and either $\theta=2^{m/2}$ and $Sz(2^{m/2})\trianglelefteq G_{0}^{\Delta _{i}} \leq \left( Z_{2^{m/2}-1} \times Sz(2^{m/2}) \right).Z_{m/2}$, or $\theta=1,2$ and $Sz(2^{m/2})\trianglelefteq G_{0}^{\Delta _{i}} \leq Sz(2^{m/2}).Z_{m/2}$;
\item a $2$-$(2^{2m},2^{m},2^{m})$ design, $m \equiv 0 \pmod{3}$, $\theta=1$ and $G_{2}(2^{m/3})\trianglelefteq G_{0}^{\Delta _{i}} \leq \left( Z_{2^{m/3}-1} \times G_{2}(2^{m/3}) \right).Z_{m/3}$;
\item a $2$-$(3^{4},3^{2},3)$ design, $\theta=3$, and $SL_{2}(5)\trianglelefteq G_{0}^{\Delta _{i}} \leq (Z_{2}.S_{5}^{-}):Z_{2}$;
\item one of the two $2$-$(2^{6},2^{3},2^{2})$-designs, $\theta =2$, and either $G_{0}$ is one of the groups $3^{1+2}:Q_{8}$, $%
3^{1+2}:Z_{8}$ or $3^{1+2}:SD_{16}$, or
 $3^{1+2}:Z_{8} \leq G_{0}^{\Delta _{i}} \leq PSU_{3}(3)$;
\item a $2$-$(2^{6},2^{3},2^{3})$-design, $\theta =1$, and $G_{0}^{\Delta _{i}}$ is one of the groups $3^{1+2}:Q_{8}$, $%
3^{1+2}:C_{8}$, $3^{1+2}:SD_{16}$, $\left( 3^{1+2}:Q_{8}\right) :3:2$.
\end{enumerate}
\end{enumerate}

\end{enumerate}
\end{theorem}

See \cite{MF,Mo1, Mo2} for a proof. 

\begin{proposition}
\label{affine}$G_{\Delta _{i}}^{\Delta _{i}}$ is of affine type and $\lambda
=p^{m}$.
\end{proposition}

\begin{proof}
Assume that $G_{\Delta _{i}}^{\Delta _{i}}$ is almost simple. Then either $%
PSL_{2}(8)\trianglelefteq G_{\Delta _{i}}^{\Delta _{i}}\leq P\Gamma L_{2}(8)$
and $\lambda =6$, or $PSL_{3}(3)\trianglelefteq G_{\Delta _{i}}^{\Delta
_{i}}\leq PSL_{3}(3):Z_{2}$ and $\lambda =12$ by Theorem \ref{Monty}.\\
Assume that the former occurs. Since $G^{\Sigma }$ acts $2$-transitively on $\Sigma $ by Proposition \ref{boom}, and $\left\vert\Sigma \right\vert =8$, one of the following holds by \cite[Section 2, (A) and (B)]{Ka85}:

\begin{enumerate}
\item $AGL_{1}(8)\trianglelefteq G^{\Sigma }\leq A\Gamma L_{1}(8)$;

\item $G^{\Sigma }\cong E_{8}:SL_{3}(2)$;

\item $PSL_{2}(7)\trianglelefteq G^{\Sigma }\leq PGL_{2}(7)$;

\item $A_{8}\trianglelefteq G^{\Sigma }\leq S_{8}$.
\end{enumerate}

Assume that (4) holds. Since $G(\Sigma )G(\Delta_{i}) \unlhd G_{\Delta_{i}}$ and  $A_{7}\trianglelefteq
G_{\Delta_{i}}^{\Sigma} \leq S_{7}$, either $G(\Delta_{i}) \trianglelefteq G(\Sigma)$ or $G(\Delta_{i})/(G(\Delta_{i})\cap G(\Sigma ))$ contains a subgroup isomorphic to $A_{7}$. The latter is ruled out by Lemma \ref{NoFTR}, since $\lambda=6$, whereas the former implies that a quotient group of $G_{\Delta _{i}}^{\Delta_{i}}$ is isomorphic to $A_{7}$, which is impossible as $PSL_{2}(8)\trianglelefteq G_{\Delta _{i}}^{\Delta _{i}}\leq P\Gamma L_{2}(8)$. Thus (4) is ruled out.\\ 
Assume that one of (1)--(3) occurs. Since $G(\Sigma )G(\Delta
_{i})\trianglelefteq G_{\Delta _{i}}$ and $PSL_{2}(8)\trianglelefteq
G_{\Delta _{i}}^{\Delta _{i}}\leq P\Gamma L_{2}(8)$, either $G(\Sigma
)\trianglelefteq G(\Delta _{i})$ or $PSL_{2}(8)\trianglelefteq G(\Sigma
)/(G(\Sigma )\cap G(\Delta _{i}))$. The former implies $G_{\Delta_{i}}^{\Delta_{i}}\cong G_{\Delta_{i}}^{\Sigma}/G(\Delta_{i})^{\Sigma}$ and hence a quotient group of $G_{\Delta_{i}}^{\Sigma}$ has a subgroup isomorphic to $PSL_{2}(8)$,
but this is clearly impossible. So $%
PSL_{2}(8)\trianglelefteq G(\Sigma )/(G(\Sigma )\cap G(\Delta _{i}))$ and $A_{7}\unlhd
G_{\Delta _{i}}^{\Sigma}$. Hence, if $W$ is any Sylow $7$-subgroup of 
$G_{\Delta _{i}}$, $7^{2} \mid \left\vert W\right\vert$. Then $7 \mid \left\vert W(\Delta _{i})\right\vert$, since $PSL_{2}(8)\trianglelefteq
G_{\Delta _{i}}^{\Delta _{i}}\leq P\Gamma L_{2}(8)$, but this contradicts Lemma \ref{NoFTR}.

Assume that $PSL_{3}(3)\trianglelefteq G_{\Delta _{i}}^{\Delta _{i}}\leq
PSL_{3}(3):Z_{2}$ and $\lambda =12$. Since $G^{\Sigma }$ acts $2$%
-transitively on $\Sigma $, with $\left\vert\Sigma \right\vert =14$, one of the following holds by \cite[Section 2, (A) and (B)]{Ka85}:

\begin{enumerate}
\item $PSL_{2}(13)\trianglelefteq G^{\Sigma }\leq PGL_{2}(13)$;

\item $A_{14}\trianglelefteq G^{\Sigma }\leq S_{14}$.
\end{enumerate}

We may proceed as the $PSL_{2}(8)$-case to rule out (1) and (2),
this time $W$ is a Sylow $13$-subgroup of $G_{\Delta _{i}}$.
\end{proof}

\bigskip

\begin{lemma}
\label{if} The following hold:

\begin{enumerate}
\item $G(\Delta _{i})\trianglelefteq G(\Sigma )\leq G_{\Delta _{i}}$ for each $i=1,...,p^{m}+2$.

\item $G(\Delta _{i})\cap G(\Delta _{j})=1$ for each $i,j=1,...,p^{m}+2$ with $i \neq j$.

\end{enumerate}
\end{lemma}

\begin{proof}
Since $G_{\Delta _{i}}$ acts transitively on $\Sigma \setminus \left\{ \Delta
_{i}\right\} $ by Proposition \ref{boom}, and since $G(\Delta _{i})\trianglelefteq G_{\Delta _{i}}$, it
follows that $\Sigma \setminus \left\{ \Delta _{i}\right\} $ is union of $G(\Delta
_{i})$-orbits of the equal length $z$, where $z$ is a divisor of $p^{m}+1$ by Proposition \ref{affine}. Assume that $z>1
$. Then $\mathcal{D}_{i}$ is a translation plane of order $p^{m}$ by Lemma %
\ref{NoFTR}. Let $U$ be a Sylow $u$-subgroup of $G(\Delta _{i})$, where $u$
is a prime divisor of $z$. Arguing as in Lemma \ref{NoFTR}, with $U$ in the
role of $W$, we see that $U$ fixes at least $p^{m}(p^{m}+1)$ blocks of $%
\mathcal{D}$ and each of these intersects $\Delta _{i}$ in $p^{m}$ points, since $\mathcal{D}$ is a translation plane and $\theta=p^{m}$. Let $B$ be any of
such blocks. Then $U$ preserves $\Delta _{i}$ and at least one the $%
p^{m}$ elements of $\Sigma \setminus \left\{ \Delta _{i}\right\} $ intersecting $B$,
say $\Delta _{j}$. Then $\left\vert \Delta _{j}^{G(\Delta _{i})}\right\vert $
is coprime to $u$, whereas $\left\vert \Delta _{j}^{G(\Delta
_{i})}\right\vert =z$ and $u\mid z$. Thus $G(\Delta _{i})$ preserves each
element of $\Sigma$ and hence $G(\Delta
_{i})\leq G(\Sigma )$. Actually, $G(\Delta _{i})\trianglelefteq G(\Sigma )$
as $G(\Sigma ) \leq G_{\Delta _{i}}$.

Let $\gamma \in G(\Delta _{i})\cap G(\Delta _{i})$, with $i \neq j$, then $\gamma $ fixes $%
2p^{2m}$ points of $\mathcal{D}$. If $\gamma \neq 1$ then $2p^{2m}\leq p^{m}(p^{m}+2)$ by 
\cite[Corollary, 3.7]{La}. So $\lambda=p^{m}\leq 2$, which is contrary to Lemma \ref{c=d}. Thus $\gamma=1$ and hence $G(\Delta _{i})\cap G(\Delta _{j})=1$ for $i \neq j$.
\end{proof}

\bigskip

\begin{corollary}\label{GTU}
$G(\Sigma) \neq 1$.
\end{corollary}

\begin{proof}
Suppose that $G(\Sigma) = 1$. Then $G(\Delta)=1$ for each $\Delta \in \Sigma$ by Lemma \ref{if}(1) and hence $Soc(G_{\Delta})$ is elementary abelian of order $p^{2m}$ by Proposition \ref{affine}. Then $\Sigma \setminus \{\Delta\}$ is partitioned in $Soc(G_{\Delta})$-orbits of equal length $p^{h}$, with $h>0$, since $G_{\Delta}$ acts transitively on $\Sigma \setminus \{\Delta\}$, $Soc(G_{\Delta}) \trianglelefteq G_{\Delta}$ and $G(\Sigma) = 1$. Then $p$ divides $\left\vert \Sigma \setminus \{\Delta\} \right\vert$, thus contradicting $\left\vert \Sigma \setminus \{\Delta\} \right\vert =p^{m}+1$. 
\end{proof}
\bigskip

\begin{proposition}
\label{qcDv} Let $V$ be a minimal normal subgroup of $G$ contained in $G(\Sigma)$.

\begin{enumerate}

\item $V^{\Delta_{i}} \cong Soc(G_{\Delta _{i}}^{\Delta _{i}})$ for each $i=1,...,p^{m}+2$.

\item $V$ is an elementary abelian $p$-group of order $p^{2m+t}$, where $0 \leq t \leq 2m$.

\item $C_{G}(V)\cap G(\Sigma)$ is an elementary abelian $p$-group order $p^{2m+y}$, where $t \leq y \leq 2m$, containing $V$.
\end{enumerate}
\end{proposition}

\begin{proof}

Let $V$ be a minimal normal subgroup of $G$ contained in $G(\Sigma)$. Then $V$ acts transitively on $\Delta_{i}$ for each $i=1,...,p^{m}+2$ by Lemma \ref{Ordine}. Moreover, $V^{\Delta_{i}} \trianglelefteq G_{\Delta _{i}}^{\Delta _{i}}$. If $R$ is a minimal normal subgroup of $G_{\Delta _{i}}^{\Delta _{i}}$ contained in $V^{\Delta_{i}}$, then $Soc(G_{\Delta
_{i}}^{\Delta _{i}})=R\trianglelefteq V^{\Delta_{i}}$ by \cite[Theorem 4.3B(i)]{DM}, since $G_{\Delta _{i}}^{\Delta _{i}}$ acts
primitively on $\Delta _{i}$ by Lemma \ref{prim} and since $Soc(G_{\Delta _{i}}^{\Delta _{i}})$
is an elementary abelian group of order $p^{m}$ by Proposition \ref{affine}. Thus $V^{\Delta _{i}}$ contains a normal subgroup isomorphic to $Soc(G_{\Delta _{i}}^{\Delta _{i}})$, and hence $V$ is an elementary abelian $p$-group, since $V$ is a minimal normal subgroup of $G$. Therefore $V^{\Delta _{i}} \cong Soc(G_{\Delta _{i}}^{\Delta _{i}})$, which is (1).\\
It follows from (1) that $\left\vert V^{\Delta _{i}}\right\vert =p^{2m}$. Let $p^{t}=\left\vert V(\Delta _{i_{0}})\right\vert $, with $t \geq 0$, for some $i_{0}$, then $\left\vert V(\Delta _{i})\right\vert =p^{t} $ for each $i=1,...,p^{m}+2$, since $%
V\trianglelefteq G$ and $G$ acts transitively on $\Sigma$. Moreover, it follows from Lemma \ref{if}(2) that, $V(\Delta _{s})\cap V(\Delta _{s^{\prime}})=1$ for each $s,s^{\prime}=1,...p^{m}+2$ with $s \neq s^{\prime}$. Thus $V(\Delta_{s})$ is isomorphic to a subgroup of $V^{\Delta_{s^{\prime}}}$. Therefore $p^{t} \leq p^{2m}$ and hence $0 \leq t \leq 2m$. So $\left\vert V\right\vert =p^{2m+t}$, with $0 \leq t \leq 2m$, and we obtain (2).\\
Set $C=C_{G}(V)$ and $K=G(\Sigma)$ and recall that $G$ is permutationally isomorphic to a subgroup of $G_{\Delta}^{\Delta} \wr G^{\Sigma}$ by \cite[Theorem 5.5]{PS}. Since $$C\cap K \leq \prod_{\Delta \in \Sigma}(C \cap K)^{\Delta} \leq  \prod_{\Delta \in \Sigma}C_{K^{\Delta}}(V^{\Delta}) = \prod_{\Delta \in \Sigma}V^{\Delta},$$ it follows $C\cap K$ is an elementary abelian $p$-group. By repeating the final part of the argument in (2) with $C \cap K$ in the role of $V$, we see that the order of $C \cap K$ is $p^{2m+y}$, where $t \leq y \leq 2m$, as $V \leq C$, and we obtain (3).

\end{proof}

\bigskip
\begin{remark}
The proof of (3) is a slight modification of an argument contained in the proof of \cite[Lemma 3.2]{LPR}: here we consider a minimal normal subgroup of $G$ instead of $O_{p}(G(\Sigma))$, the largest normal $p$-subgroup of $G(\Sigma)$. In this way $G/C_{G}(V)$ is an irreducible subgroup of $GL(V)$, and this will play an important role in completing the proof of Theorem \ref{t1}.
\end{remark}

\bigskip

\begin{corollary}\label{glue}
Let $\Delta \in \Sigma $ and $x \in \Delta $, then $G^{\Sigma}_{\Delta}$ is isomorphic to a quotient group of $G_{x}^{\Delta}$.
\end{corollary}

\begin{proof}
It follows from Lemma \ref{if}(1) and Proposition \ref{qcDv}(1) that $G_{\Delta}=G(\Sigma)G_{x}$ with $G(\Delta) \trianglelefteq G(\Sigma)_{x}$. Thus 
\begin{equation}\label{isom}
G^{\Sigma}_{\Delta} \cong G_{x}/G(\Sigma)_{x} \cong G_{x}^{\Delta}/G(\Sigma)_{x}^{\Delta}
\end{equation}
which is the assertion. 
\end{proof}

\begin{lemma}
\label{div}If $u=p^{m}+2$, with $u$ prime and $u\geq 5$, divides $p^{z}-1$, where $0<z\leq 4m$,
then either $(p^{m},u,z)=(3,5,4)$ or $(p^{m},u,z)=(9,11,5)$.
\end{lemma}

\begin{proof}
Clearly $z>m$. Set $z=m+w$, where $w > 0$. Then $u$ divides $p^{w}(p^{m}+2)-2p^{w}-1$ and
hence $2p^{w}+1$. Then $p^{w}\geq p^{m}/2+1/2$ and hence $w=m+y$ for some $y\geq 0$, as $p$
is odd. Thus $u$ divides $%
2p^{y}(p^{m}+2)-(4p^{y}-1)$ and hence $4p^{y}-1$. Then either $p^{m}+2=4p^{y}-1$ or $2(p^{m}+2)\leq 4p^{y}-1$. The former yields $p^{y}=3$, $%
p^{m}=3^{2}$ and hence $(p^{m},u,z)=(9,11,5)$, the latter implies $p^{y}\geq p^{m}/2+5/4$ and hence $y=m+x$ for some $0\leq x \leq 2m$, since $z \leq 4m$. Therefore $u$ divides $(4p^{m}+8)p^{x}-(8p^{x}+1)$ and hence $8p^{x}+1$. Thus $%
a(p^{m}+2)=8p^{x}+1$ for some $a\geq 1$. If $x<m$, then $%
(ap^{m-x}-8)p^{x}+2a-1=0$. Therefore $ap^{m-x}<8$ and hence
either $a=1$ and $p^{m-x}=3,5$ or $7$, or $a=2$ and $p^{m-x}=3$, since $p$
is odd. It is easy to verify that no solutions arise in these cases, since $u$ is a prime. Thus $x=m$ and hence $%
(8-a)p^{m}=2a-1$, with $1\leq a<8$. Thus either $a=3$ and $p^{m}=3$, or $a=7$
and $p^{m}=13$. However the latter cannot occur since $u=p^{m}+2$ and $u$ is a prime, hence $(p^{m},u,z)=(3,5,4)$.    
\end{proof}

\bigskip

\begin{proposition}\label{happens} One of the following holds:
\begin{enumerate}
\item $C_{G}(V) \leq G(\Sigma)$.
\item $C_{G}(V)=V \times U$, where $U$ is a minimal normal subgroup of $G$ of order $u^{h}$, with $u$ prime and $h \geq 1$, satisfying the following properties:
\begin{enumerate}
\item[(a)] $u^{h}=p^{m}+2$;
\item[(b)] The $U$-orbit decomposition of the point set of $\mathcal{D}$ is a $G$-invariant partition $\Sigma^{\prime}=\left\lbrace \Delta_{1}^{\prime},...,\Delta_{p^{2m}}^{\prime} \right\rbrace$ such that $\left\vert \Delta_{j}^{\prime} \right \vert =p^{m}+2$ for each $j=1,...,p^{2m}$ and the following hold:
\begin{enumerate}
\item [(i)] For each $B\in \mathcal{B}$ and $\Delta _{j}^{\prime}\in \Sigma^{\prime} $, the size $\left\vert B\cap \Delta _{j}^{\prime}\right\vert $ is either $0$ or $2$.
\item [(ii)] $\left\vert \Delta _{i}\cap \Delta _{j}^{\prime}\right\vert =1$ for each $\Delta _{i}\in \Sigma$ and $\Delta _{j}^{\prime}\in \Sigma^{\prime} $;
\item [(iii)] $G_{\Delta _{j}^{\prime}}^{\Delta _{j}^{\prime}}$ acts $2$-transitively on $\Delta_{j}^{\prime}$ for each $\Delta _{j}^{\prime}\in \Sigma^{\prime} $;
\item[(iv)] $U^{\Delta _{j}^{\prime}}=Soc(G_{\Delta _{j}^{\prime}}^{\Delta _{j}^{\prime}})$ for each $\Delta _{j}^{\prime}\in \Sigma^{\prime} $.
\end{enumerate} 
\end{enumerate}
\end{enumerate}
\end{proposition}

\begin{proof}
Let $C=C_{G}(V)$ and assume that $C\nleq G(\Sigma )$. Then $1\neq C^{\Sigma }\trianglelefteq G^{\Sigma }$. Thus 
$Soc(G^{\Sigma })\leq C^{\Sigma }$, as $G^{\Sigma }$ acts $2$-transitively on $\Sigma $. Therefore, $C^{\Sigma }$ acts transitively on $\Sigma$.  

Assume that $(C^{\Sigma })_{\Delta _{1}}\neq 1$, where $\Delta _{1} \in \Sigma$. Then $V<C_{\Delta _{1}}$ as $%
V \leq C\cap G(\Sigma )$ by Proposition \ref{qcDv}(3). So $C_{\Delta _{1}}=C_{x}V$, and hence $C_{x}\leq
G(\Delta _{1})$, since $V$ acts transitively on $\Delta _{1}$ and $C=C_{G}(V)$. Therefore, $C_{x} \leq G(\Sigma )$ by Lemma \ref{if}(1). So $C_{\Delta_{1}} \leq G(\Sigma)$, and hence $(C^{\Sigma })_{\Delta _{1}}= 1$, a contradiction.

Assume that $(C^{\Sigma })_{\Delta _{1}}=1$. Thus $C^{\Sigma }$ acts regularly on $\Sigma $ and hence $C^{\Sigma }=Soc(G^{\Sigma })$, since $Soc(G^{\Sigma })\leq C^{\Sigma }$. Therefore $C/(C \cap G(\Sigma))\cong Soc(G^{\Sigma })$, where $\left\vert Soc(G^{\Sigma })\right\vert =p^{m}+2$, with $p^{m}+2\equiv 1,2\pmod{4}$ according to whether $p$ is odd or even respectively, since $p^{m} \geq 3$ by Lemma \ref{c=d}. Thus $Soc(G^{\Sigma })$ is an elementary abelian group of order $u^{h}$, for some prime  $u$ and some integer $h \geq 1$, since $G^{\Sigma }$ acts $2$-transitively on 
$\Sigma $. Therefore $u^{h}=p^{m}+2$ and hence $%
u\neq p$. Then $C=X: U$, where $X=C \cap G(\Sigma)$ and $U$ is an elementary abelian of order $u^{h}$, by \cite[Theorem 6.2.1]{Go}
since $C/(C \cap G(\Sigma))\cong Soc(G^{\Sigma })$ and since $X$ is an elementary abelian $p$-group by Proposition \ref{qcDv}(3). In particular $U \leq GL(X)$. Then $X=X_{1} \oplus \cdots \oplus X_{\ell}$, with $\ell \geq 1$, where the $X_{s}$'s, $s=1,...,\ell$, are irreducible $U$-invariant subspaces of $X$ by \cite[Theorem 3.3.1]{Go}. Moreover, for each $s=1,...,\ell$ there is a subgroup of $U_{s}$ of $U$ of index at most $u$, fixing $X_{s}$ pointwise by \cite[Theorem 3.2.3]{Go}, since $U$ is elementary abelian. \\
Assume that $V<X$. Then there is $X_{s_{0}}$ containing an element, say $x$, such that $x \notin V$. If $h>1$, let $\psi$ an element of $U_{s_{0}}$, $\psi \neq 1$. Then $\psi$ centralizes $V$, $\langle x \rangle$ and hence $V\oplus \langle x \rangle$. Now $V\oplus \langle x \rangle$ acts on $\Delta_{1}$, and since it contains $p^{2m+1}$ elements, there is an element $\alpha \in V\oplus \langle x \rangle$, $\alpha \neq 1$, such that $\alpha \in G(\Delta_{1})$, and $\alpha$ and $\psi$ commute. On the other hand $U$ acts regularly on $\Sigma$, since $U^{\Sigma}=C/(C \cap G(\Sigma))\cong Soc(G^{\Sigma })$, with $C \cap G(\Sigma)$ a $p$-group by Proposition \ref{qcDv}(3). Hence, $\psi$ maps $\Delta_{1}$ onto $\Delta_{e}$ for some $e>1$. So $\alpha \in G(\Delta_{1}) \cap G(\Delta_{e})$, with $\alpha \neq 1$, since $\alpha$ and $\psi$ commute, and this contradicts Lemma \ref{if}(2). Thus $h=1$ and hence $U \cong Z_{u}$. If $U$ fixes an element in $X \setminus V$, we reach a contradiction by using the previous argument. Thus $U$ does not fix points in $X \setminus V$ and hence $u\mid p^{2m+y}-p^{2m+t}$, with $0\leq t < y \leq 2m$ by Lemma \ref{qcDv}(2),(3), since $V<X$. Then $u \mid p^{y-t}-1$, with $0 <y-t \leq 2m$, which is impossible for Lemma \ref{div}. Thus $X=V$ and hence $C=V \times U$, where $U$ is elementary abelian of order $u^{h}$, with $u^{h}=p^{m}+2$. Moreover $U\trianglelefteq G$ as $C\trianglelefteq G$.\\
Let $U^{\ast }$ be a minimal normal subgroup of $G$
contained in $U$. The set $\Sigma^{\prime}$ of all point-$U^{\ast}$-orbits is $G$-invariant partition of the point set of $\mathcal{D}$. If either $\Sigma^{\prime}=\Sigma$ or $U^{\ast}$ acts point-transitively on $\mathcal{D}$, then $c$ or $v$ is a power of $u$ respectively. However both these cases lead to a contradiction, since $c=p^{2m}$, $v=p^{2m}(p^{m}+2)$ and $u\neq p$. Thus, $\Sigma^{\prime}=\left\lbrace \Delta_{1}^{\prime},...,\Delta_{p^{2m}}^{\prime} \right\rbrace$, with $\left\vert \Delta_{j}^{\prime} \right \vert =p^{m}+2=u^{h}$ for each $j=1,...,p^{2m}$, by Lemma \ref{Ordine}. Hence $U^{\ast}=U$ and we obtain (2a). Moreover, for each $B\in \mathcal{B}$ and $\Delta _{j}^{\prime}\in \Sigma^{\prime} $, the size $\left\vert B\cap \Delta _{j}^{\prime}\right\vert $ is either $0$ or $2$, $G_{\Delta _{j}^{\prime}}^{\Delta _{j}^{\prime}}$ acts $2$-transitively on $\Delta_{j}^{\prime}$ and $U^{\Delta _{j}^{\prime}}=Soc(G_{\Delta _{j}^{\prime}}^{\Delta _{j}^{\prime}})$. Finally, $\left\vert \Delta _{i}\cap \Delta _{j}^{\prime}\right\vert =1$ for each $\Delta _{i}\in \Sigma$ and $\Delta _{j}^{\prime}\in \Sigma^{\prime} $. Thus (2b.i)--(2b.iv) follow.
\end{proof}

\bigskip

The Diophantine equation in Proposition \ref{happens}(2.a) is a special case of the Pillai Equation (e.g see \cite{SS}). It has at most one solution in positive integers $(m,h)$ by \cite[Theorem 6]{SS}. Moreover, $p>3$ for $h>1$, and $(p^{m},u^{h})=(5^{2},3^{3})$ for $h>1$ and $m$ even by \cite[Lemmas 2 and 4]{SS}. 

\bigskip

\section{Reduction to the case $C_{G}(V) \leq G(\Sigma)$}

In this section we show that only (1) of Proposition \ref{happens} occurs. Hence, assume that $C_{G}(V)=V \times U$, where $U$ is a minimal normal elementary abelian $u$-subgroup of $G$ of order $u^{h}$, where $u^{h}=p^{m}+2$, satisfying properties (2b.i)--(2b.iv) of Proposition \ref{happens}. 

\bigskip

\begin{lemma}\label{quotient}
Let $\Delta \in \Sigma$ and $\Delta^{\prime} \in \Sigma^{\prime}$ and let $x$ be their intersection point. Then the following hold:

\begin{enumerate} 
\item $G_{x}$ acts faithfully on $\Delta$;

\item $G_{x}^{\Delta^{\prime }}$ is isomorphic to a quotient group of $G_{x}$.
\end{enumerate}
\end{lemma}

\begin{proof}
Let $\Delta _{i} \in \Sigma $ and $\Delta
_{j}^{\prime }\in \Sigma ^{\prime }$ and let $x_{ij}$ be their (unique)
intersection point.
Since $G(\Delta) \leq G(\Sigma)$ by Lemma \ref{if}(1), $G(\Delta)$ preserves each $\Delta _{i}$. On the other hand, $G(\Delta)$ preserves each $\Delta _{j}^{\prime}$ since these ones intersect $\Delta$ in a unique point and since $\Sigma^{\prime}$ is a $G$-invariant partition of the point set of $\mathcal{D}$. Thus $G(\Delta)$ fixes each $x_{ij}$, that is, $G(\Delta)$ fixes $\mathcal{D}$ pointwise. Hence $G(\Delta)=1$, which is (1). Finally, assertion (2) holds since $G(\Delta^{\prime}) \trianglelefteq G_{x}$.
\end{proof}

\begin{lemma}\label{mamma}
The following hold:

\begin{enumerate}
\item $\frac{p^{m}}{\theta }(p^{m}+1)$ divides $\left\vert G_{x}\right\vert $;

\item $\theta (p^{m}+2)(p^{m}+1)$ divides $\left\vert G_{\Delta^{\prime
}}^{\Delta ^{\prime }}\right\vert $.
\end{enumerate}
\end{lemma}

\begin{proof}
Let $\Delta \in \Sigma$ and $\Delta^{\prime} \in \Sigma^{\prime}$ and let $x$ be their intersection point. The replication number $r=\frac{p^{m}}{\theta }(p^{m}+1)$ divides $%
\left\vert G_{x}\right\vert $, since $\mathcal{D}_{i}
$ is a flag-transitive $2$-$(p^{2m},p^{m},\frac{p^{m}}{\theta })$ design, hence (1) holds.

Since $G_{\Delta^{\prime }}^{\Delta^{\prime }}$ acts $2$%
-transitively on $\Delta^{\prime}$ and $U^{{\Delta}^{\prime}}=Soc(G_{\Delta^{\prime }}^{\Delta^{\prime }})$, with $U^{\Delta^{\prime}}$ elementary abelian of order $p^{m}+2$, by Proposition \ref%
{happens}(2a), (2.b.iii) and (2.b.iv), it follows that $(p^{m}+2)(p^{m}+1)$ divides $%
\left\vert G_{\Delta^{\prime }}^{\Delta^{\prime }}\right\vert $.

Let $B$ be any block of $\mathcal{D}$ such that $x\in B$. Then $B\cap \Delta
^{\prime }=\left\{ x,y\right\} $, for some $y\neq x$, by Proposition \ref{happens}(2.b.i). Let $\gamma \in G_{x,B\cap \Delta}\cap G(\Delta
^{\prime })$. Then $B^{\gamma }\cap \Delta
=B\cap \Delta$, and $y\in B^{\gamma }$ since $B\cap \Delta
^{\prime }=\left\{ x,y\right\} $. Thus $\left( B\cap \Delta\right) \cup \left\{
y\right\} \subseteq B^{\gamma }\cap B$, where $y\notin B\cap \Delta$ as 
$y\in \Delta^{\prime }$, $y\neq x$, and $\Delta\cap \Delta
^{\prime }=\left\{ x\right\} $. Therefore $\left\vert B^{\gamma }\cap
B\right\vert \geq \lambda +1$ and hence $B^{\gamma }=B$. So $\gamma \in
G_{x,B}$, as $x^{\gamma }=x$. Thus $G_{x,B\cap
\Delta}\cap G(\Delta^{\prime })\leq G_{x,B}$ and hence $\theta 
$ divides $\left\vert \left( G_{x,B\cap \Delta}\right) _{\Delta
^{\prime }}^{\Delta^{\prime }}\right\vert $, since $\left[
G_{x,B\cap \Delta}:G_{x,B}\right] =\theta $ by Corollary \ref{ermo}.
So, $\theta $ divides $\left\vert G_{\Delta^{\prime }}^{\Delta
^{\prime }}\right\vert $. Hence, $\theta (p^{m}+2)(p^{m}+1)\mid \left\vert G_{\Delta^{\prime
}}^{\Delta^{\prime }}\right\vert $, which is (2), since $\theta \mid p^{m}$ and $(p^{m}+2)(p^{m}+1)$ divides $\left\vert G_{\Delta^{\prime
}}^{\Delta^{\prime }}\right\vert $.
\end{proof}

\begin{theorem}\label{centraliz}
$C_{G}(V) \leq G(\Sigma)$.
\end{theorem}
\begin{proof}
Let $\Delta_{i} \in \Sigma$ and $\Delta^{\prime} \in \Sigma^{\prime}$ and let $x_{i}$ be their intersection point. Recall that $\mathcal{D}_{i}$ is isomorphic to one of the $2$-designs listed in Theorem \ref{Monty}. Since $p^{m}+2=u^{h}$, with $u$ prime, by Proposition \ref{happens}(2.a), and since $p^{m} \geq 3$ by Lemma \ref{c=d}, it follows that $p>2$. Thus (2.a.ii), (2.c.iv)--(2.c.v) and (2.c.vii)--(2.c.viii) are ruled out. Also, if $p^{m}=3^{2}$ or $3^{3}$ then $h=1$ and either $u=11$ or $u=29$ respectively. Hence $G_{\Delta^{\prime}}^{\Delta^{\prime}} \cong AGL_{1}(u)$ by Proposition \ref{happens}(2.b.iii). Then $\theta=1$ by Lemma \ref{mamma}(2), and hence (2.a.iii)--(2.a.iv) and (2.c.vi) of Theorem \ref{Monty} are ruled out. 

Assume that (2.c.i) occurs. Then $p^{m}+2 \mid (p^{m/2}+1)(p^{m/2}-1)^{2}m$ by Lemma \ref{mamma}(2). Since $p^{m}+2$ is the power of an odd prime $u$, it follows that either $p^{m}+2 \mid (p^{m/2}+1)m/2$ or $p^{m}+2 \mid(p^{m/2}-1)^{2}m/2$. In both case we get $p^{m}+2 \leq (p^{m/2}-1)^{2}p^{m/2}$, since $m/2 \leq p^{m/2}$, and we reach a contradiction.

Assume that (2.c.ii) occurs. Arguing as above, we see that either $p^{m}+2 \mid (p^{m/3}+1)m/3$ or $p^{m}+2 \mid(p^{m/3}-1)^{2}m/3$, which is impossible since $m/3 \leq p^{m/3}$.

Assume that (2.c.iii) occurs. Then either $p^{m}+2 \mid (p^{m/2}+1)^{2}m/2$ or $p^{m}+2 \mid(p^{m/2}-1)^{3}m/2$. Since $p>2$, it follows that $m/2 \leq p^{m/2}-2$ and hence the former is ruled out. Since $(p^{m}+2,p^{m/2}-1) \mid 3$ forces $p^{m}+2=3^{h}$, with $h>1$, and since $m$ is even, it results $m=2$ and $p^{m/2}=5$ by \cite[Lemmas 4]{SS}. However, such values do not fulfill $p^{m}+2 \mid(p^{m/2}-1)^{3}m/2$. So this case is excluded. Therefore only cases (2.a.i) or (2.b) are admissible, and bearing in mind Lemma \ref{quotient}(1), one of the following holds:
\begin{enumerate}
\item[(i).] $\mathcal{D}_{i} \cong AG_{2}(p^{m})$, $p>2$ and $p^{m} \neq 3^{2},3^{3}$, $\lambda=\theta=p^m$ and $G_{x_{i}} \leq \Gamma L_{2}(p^{m})$.
\item[(ii).] $\mathcal{D}_{i}$ is a $2$-$(p^{2m},p^{m}, p^{m-s})$ design, $p>2$ and $p^{m} \neq 3^{2},3^{3}$, $\theta=p^{s}$ with $0 \leq s \leq m$, the blocks are subspaces of $AG_{2m}(p)$ and $G_{x_{i}} \leq \Gamma L_{1}(p^{2m})$.
\end{enumerate}
Assume that (i) occurs. Then $p^{m}$ divides the order of $G_{x_{i}}^{\Delta^{\prime}}$, and hence that of $G_{x_{i}}$ by Lemmas \ref{quotient}(2) and \ref{mamma}(2), since $\theta=p^{m}$ and $[G_{\Delta^{\prime}}^{\Delta^{\prime}}:G_{x_{i}}^{\Delta^{\prime}}]=p^{m}+2$, with $p>2$. This fact, together with Lemma \ref{mamma}(1), implies $SL_{2}(p^{m})\trianglelefteq G_{x_{i}} \leq \Gamma L_{2}(p^{m})$. Then $G_{x_{i}}^{\Delta^{\prime}}$ contains either $PSL_{2}(p^{m})$ or $SL_{2}(p^{m})$ as a normal subgroup by Lemma \ref{quotient}(2), since $p^{m}$ divides the order of $G_{x_{i}}^{\Delta^{\prime}}$. Then $p^{m}=3,5,13$ by \cite[Section 2, (B)]{Ka85}, since $p>2$ and $p^{m} \neq 3^{2}$, since $G_{\Delta^{\prime}}^{\Delta^{\prime}}$ is an affine $2$-transitive group by Proposition \ref{happens}(2b.iii)--(2b.iv). Actually $p^{m}\neq 13$ and $u^{h}=5,7$ for $p^{m}=3,5$ respectively, since $u^{h}=p^{m}+2$ and $u$ is a prime. The case $p^{m}=3$ cannot occur by \cite[Proposition 5.1]{P}, the case $p^{m}=5$ is excluded since $G_{\Delta^{\prime}}^{\Delta^{\prime}}\cong AGL_{1}(5)$ in these cases, whereas $G_{x_{i}}^{\Delta^{\prime}}$ is non-solvable. Thus case (i) is ruled out.\\ 
Assume that (ii) occurs. Then $p^{m-s}$ divides the order of $G_{x_{i}}$ by Lemma \ref{mamma}(1). On the other hand, $p^{s}$ divides the order of $G_{x_{i}}^{\Delta^{\prime}}$, and hence that of $G_{x_{i}}$, by Lemmas \ref{quotient}(2) and \ref{mamma}(2). Hence $p^{a}$ divides the order of $G_{x_{i}}$ and hence that $\Gamma L_{1}(p^{2m})$, where $a=\max\{s,m-s\} \geq m/2$. So $p^{m/2} \leq m$, which is a contradiction as $p$ is odd. Thus (2) of Proposition \ref{happens} is ruled out and hence the assertion follows.
\end{proof}

\section{Proof of Theorem \ref{t1}}
In this section we complete the proof of Theorem \ref{t1}. In the sequel, $C_{G}(V)$ and $G/C_{G}(V)$ will simply be denoted by $C$ and $H$ respectively.
\bigskip

\begin{proposition}\label{irre}
$H$ is an irreducible subgroup of $GL_{2m+t}(p)$, where $0 \leq t \leq 2m$.
\end{proposition}

\begin{proof}
Since $V$ is a minimal normal elementary abelian subgroup of $G$ of order $p^{2m+t}$, with $0 \leq t \leq 2m$, by Proposition \ref{qcDv}(2), the assertion follows.   
\end{proof}
\bigskip

For each divisor $n$ of $2m+t$ the group $\Gamma L_{n}(p^{(2m+t)/n})$ has a
natural irreducible action on $V$. By Proposition \ref{irre} we may choose $n$ to be minimal such that 
$H\leq \Gamma L_{n}(p^{(2m+t)/n})$ in this action and write $%
q=p^{(2m+t)/n}$. 

\bigskip

Let $a,e$ be integers. A divisor $w$ of $a^{e}-1$ that is coprime to each $a^{i}-1$ for $1 \leq i<e$ is
said to be a \emph{primitive divisor}, and we call the largest primitive
divisor $\Phi _{e}^{\ast }(a)$ of $a^{e}-1$ the \emph{primitive part} of $%
a^{e}-1$. One should note that $\Phi _{e}^{\ast }(a)$ is strongly related to
cyclotomy in that it is equal to the quotient of the cyclotomic number $\Phi
_{e}(a)$ and $(n,\Phi _{e}(a))$ when $e>2$. Also $\Phi _{e}^{\ast }(a)>1$
for $a \geq 2$, $e>2$ and $(q,e)\neq (2,6)$ by Zsigmondy's Theorem (for instance, see 
\cite[P1.7]{Rib}).

\bigskip
Since $G^{\Sigma }$ is a $2$-transitive almost simple group by Proposition \ref{boom}, either $G^{\Sigma}$ is of affine type or an almost simple group. We analyze the two cases separately.
\bigskip

\subsection{$G^{\Sigma}$ is of affine type}
In this subsection we assume that $G^{\Sigma }$ is of affine type. Hence $Soc(G^{\Sigma })$ is an elementary abelian $u$-group for some prime $u$. Let $u^{h}$ be the order of $Soc(G^{\Sigma })$, then $u^{h}=p^{m}+2$ as $\left \vert \Sigma \right \vert =p^{m}+2$. In the sequel $U$ will denote the pre-image of $Soc(G^{\Sigma })$ in $G$.

\begin{lemma} \label{Phistar}
The following hold:

\begin{enumerate}
\item A quotient group of $H$ has a $2$-transitive permutation
representation of degree $q^{n/2}+2$.

\item $(q^{n/2}+1)\left( q^{n/2}+2\right) \mid \left\vert H\right\vert $. In
particular $\Phi _{2m}^{\ast }(p)\mid \left\vert H\right\vert $.

\item Either $\Phi _{2m}^{\ast }(p)>1$, or $(p^{m},u^{h})=(3,5)$ or $%
(7,9)$.
\end{enumerate}
\end{lemma}

\begin{proof}
Since $C \leq G(\Sigma)$ by Theorem \ref{centraliz} and since $G^{\Sigma }$ acts $2$-transitively on $\Sigma $ by Proposition \ref{boom}, a quotient of $H$
is isomorphic to $G^{\Sigma }$. Thus (1) and (2) follow.

Suppose that $\Phi _{2m}^{\ast }(p)=1$, then either $n=2$ and $q$ is a
Mersenne prime, or $(q,n)=(2,6)$ by \cite[ P1.7]{Rib}. The latter is
ruled out since $q^{n/2}+2=10$ in this case, the former yields $u^{h}=q^{n/2}+2=2^{y}+1$,
for some prime $y\geq 0$. Then either $(p^{m},u^{h},y)=(7,9,3)$, or $h=1$ by \cite[
B1.1]{Rib}. If $h=1$, then $u$ is a Fermat prime and hence $y=2$, as $y$
is a prime and we obtain $(u^{h},y)=(5,2)$. Thus either $(p^{m},u^{h})=(3,5)$ or $(p^{m},u^{h})=(7,9)$.
\end{proof}

\bigskip

From now on, we assume that $\Phi
_{2m}^{\ast }(p)>1$. The cases $(p^{m},u^{h})= (3,5),(7,9)$ are tackled at the end of this subsection.

\bigskip

\begin{lemma}
\label{n>1}$n>1$.
\end{lemma}

\begin{proof}
Suppose that $n=1$. Then $H\leq \Gamma L_{1}(q)$ and hence $%
(p^{m}+2)\Phi _{2m}^{\ast }(p)\mid (p^{2m+t}-1)\cdot (2m+t)$. Then $2m\mid
2m+t$ by \cite[Proposition 5.2.15(i)]{KL} and hence either $t=0$ or $t=2m$, since $0\leq t\leq 2m$. If $t=0$,
then $p^{m}+2\mid (p^{2m}-1)\cdot 2m$ and hence $p^{m}+2\mid 3m$, which is
impossible since $p^{m}> 3$ by our assumption. Thus $t=2m$ and hence $(p^{m}+2)\mid (p^{4m}-1)\cdot 4m$. So $%
p^{m}+2\mid 15m$ and hence $p^{m}=3$, but this contradicts our assumption.
\end{proof}

\begin{proposition}
\label{primpart} Let $H^{\ast}=H\cap GL_{n}(q)$, then $\Phi _{2m}^{\ast }(p)\frac{p^{m}+2}{(p^{m}+2,(2m+t)/n)}%
\mid \left\vert H^{\ast}\right\vert $.
\end{proposition}

\begin{proof}
By Lemma \ref{Phistar}(2) $ \frac{\Phi _{2m}^{\ast }(p)(p^{m}+2)}{(\Phi _{2m}^{\ast }(p)(p^{m}+2),(2m+t)/n)}\mid \left\vert H^{\ast}\right\vert $. Assume that $(\Phi _{2m}^{\ast }(p),(2m+t)/n))>1$. Then
there is a primitive prime divisor $w$ of $p^{2m}-1$ dividing $(2m+t)/n$. So $w\mid
2m+t$. On the other hand, $w=2ma+1$ for some $a\geq 1$ by \cite[Proposition 5.2.15(ii)]{KL}. Hence $(2ma+1)s=2m+t$ for some $s\geq 1$ and hence $t=s=a=1$. So $w=n=2m+1$ and $q=p$, whereas $w \mid (2m+1)/n$, a contradiction. Thus $(\Phi _{2m}^{\ast }(p),(2m+t)/n))=1$ and the assertion follows, since $(\Phi _{2m}^{\ast }(p),p^{m}+2)=1$.
\end{proof}

\begin{proposition}
\label{dec}$n=4m$, $q=p$ and one of the following holds:

\begin{enumerate}
\item $H\leq GL_{2m}(p)\wr Z_{2}$ and $H$ preserves a sum decomposition $%
V=V_{1}\oplus V_{2}$ with $\dim V_{1}=\dim V_{2}=2m$.

\item $H\leq GL_{2}(p)\circ GL_{2m}(p)$, $m>1$, and $H$ preserves a tensor
decomposition $V=V_{1}\otimes V_{2}$ with $\dim V_{1}=2$ and $\dim V_{2}=2m$.
\end{enumerate}
\end{proposition}

\begin{proof}
Assume that $Y\trianglelefteq H$, where $Y$ is isomorphic to one of the
groups $SL_{n}(q)$, $Sp_{n}(q)$ for $n$ even, $SU_{n}(q^{1/2})$ for $q$ square, or $%
\Omega _{n}^{\varepsilon }(q)$ with $\varepsilon =\pm $ for $n$ even, and $%
\varepsilon =\circ $ for $nq$ odd. Then $H \leq (Z \circ Y).Out(\bar{Y})$, where $Z=Z(GL_{n}(q))$ and $\bar{Y}=Y/(Y \cap Z)$, and hence $\left\vert H/Y \right \vert \mid (q-1)\left\vert Out(\bar{Y}) \right \vert$. \\ 
Let $X$ be the pre-image of $Y$ in $G$. Then $X \trianglelefteq G $ and hence $X^{\Sigma} \trianglelefteq G^{\Sigma}$. Either $X^{\Sigma}=1$ or $U^{\Sigma} \trianglelefteq X^{\Sigma}$, since $G^{\Sigma}$ acts $2$-transitively on $\Sigma$ and since $U^{\Sigma}=Soc(G^{\Sigma})$. The latter implies that $U^{\Sigma}$ is isomorphic to a normal subgroup of $X/(X \cap G(\Sigma))$, with $C \leq X \cap G(\Sigma)$ by Theorem \ref{centraliz}. So $X/(X \cap G(\Sigma))$ is a quotient group of the classical group $Y$ and contains a normal elementary abelian group of order $u^{h}$. Then $C \leq X \cap G(\Sigma) \leq W$, where $W$ is the pre-image of $Z(Y)$ in $G$, and hence the normal elementary abelian subgroup of order $u^{h}$ of $X/(X \cap G(\Sigma))$ is contained in $W/(X \cap G(\Sigma))$. Thus $u^{h}$ divides $\left\vert Z(Y) \right\vert$ and hence $q-1$, which is impossible since $u^{h}=q^{n/2}+2 \geq 5$ by Lemma \ref{c=d}. Therefore $X^{\Sigma}=1$ and hence $X \leq G(\Sigma)$. Thus $Y \leq G(\Sigma)/C$ by Theorem \ref{centraliz}. On the other hand, we have $G^{\Sigma} \cong H/(G(\Sigma)/C)$, where $H/(G(\Sigma)/C)$ is a quotient group of $H/Y$. Thus $\left\vert
G^{\Sigma }\right\vert$ divides the order of $H/Y$ and hence $\left\vert
G^{\Sigma }\right\vert \mid (q-1)\left\vert Out(\bar{Y}) \right \vert$.\\
Note that $\left\vert
\bar{Y}\right\vert \mid 4\cdot n \cdot \mu (2m+t)/n$, where $\mu =2$ or $3$ according to
whether $X$ is isomorphic or not to $\Omega _{8}^{+}(q)$
respectively. Then 
\begin{equation}
\Phi _{2m}^{\ast }(p)(p^{m}+2)\mid \mu \left( p^{\frac{2m+t}{n}}-1\right)
(2m+t)  \label{C8}
\end{equation}%
since $\Phi _{2m}^{\ast }(p)(p^{m}+2)$ divides the order $G^{\Sigma}$ by Proposition \ref{boom} and since $\Phi _{2m}^{\ast }(p)$ and $p$ are odd. Let $w$ be a primitive prime divisor of $p^{m}+1$.
If $w\mid p^{\frac{2m+t}{n}}-1$, then $2m\mid \frac{2m+t}{n}$ by \cite[Proposition 5.2.15(i)]{KL}, and hence $n=2$ and $t=2m$ as $t\leq 2m$ and as $n>1$ by Lemma \ref{n>1}.
Thus $\mu =2$ and $p^{m}+2\mid m\left( p^{m}-1\right) $ and hence $p^{m}+2\mid 3m$, since $p$ is odd, which is impossible by $p^{m} \geq 3$. Thus $w\mid \mu \left( 2m+t\right) $.

If $\mu =2$, then $w\mid 2m+t$. If $\mu =3$, then $n=8$ and hence $m\geq 2$
since $n\mid 2m+t$ with $t\leq 2m$. Moreover, $w\equiv 1\pmod{2m}$ by \cite%
[Proposition 5.2.15(ii)]{KL} and hence $w\neq 3$. Thus $w\mid 2m+t$ also in this
case. Again from \cite[Proposition 5.2.15(ii)]{KL} it results that $w=2m+1$ and $t=1$. Moreover $n=2m+1$, and hence $q=p$, since $n\mid 2m+1$, with $2m+1$ prime, and since $n>1$ by Lemma \ref{n>1}. Then $\Phi _{2m}^{\ast }(p)(p^{m}+2)\mid (p-1)\mu n$ by (\ref{C8}), with $n$ dividing $\Phi _{2m}^{\ast }(p)$, hence $p^{m}+2 \mid 3 \mu$, which is impossible since $p$ is odd and since $p^{m} \neq 7$ by our assumption. Thus $Y\ntrianglelefteq H$ and hence $H$ lies either in a member of $\mathcal{C}%
_{i}(\Gamma)$ for some $i$ such that $1\leq i\leq 7$, or is a member of $\mathcal{S}(\Gamma)$, where $\Gamma$ denotes $\Gamma L_{n}(q)$,  by the Aschbacher's
Theorem (see \cite{KL}).

Assume that $H$ lies in $\mathcal{S}(\Gamma)$. Then $S\trianglelefteq
H/(H\cap Z)\leq Aut(S)$, where $S$ is a \ non-abelian simple group and $Z$ is
the centre of $GL_{n}(q)$. Then the pre-image $N$ of $S$ in $H$ is
absolutely irreducible and $N$ is not a classical group over a subfield of $%
GF(q)$ in its natural representation. Let $M$ and $Q$ be the pre-images in $G$ of $N$ and of $Z(N)$ respectively. Since $M,U\trianglelefteq G$, we may use the above argument with $M$, $Q$, $N$ and $Z(N)$ in the role of $X$, $W$, $Y$ and $Z$, respectively, to obtain that either a quotient of $Z(N)$ contains a normal elementary subgroup of order $u^{h}$, or $M \trianglelefteq G(\Sigma)$ and $\left\vert
G^{\Sigma }\right\vert \mid (q-1)\left\vert Out(S) \right \vert$. The former implies that there is a subgroup of the Schur multiplier of $S$ containing a normal elementary subgroup of order $u^{h}$. Then $h=1$ by \cite[Theorem 5.1.4]{KL}. Let $\psi $ be an element of $Z(N)$ of order $u$. Then $\psi $ does not fix non-zero vectors of $V$ and so $u\mid
p^{2m+t}-1$, with $0<2m+t \leq 4m$. Thus $p^{m}=9$ by Lemma \ref{div}, since $p^{m}>3$ by our assumption, and hence $u=11$ divides $\left\vert Z(N) \right \vert$ and $n \leq 4m= 8$. Also, $H$ is an irreducible subgroup of $\Gamma L_{n}(3^{(4+t)/n})$ of order divisible by $55$ by Lemma \ref{Phistar}(2). Then either $t=1$, $n=5$, $q=3$ and $H=N$, where $N$ is isomorphic to $PSL_{2}(11)$ or $M_{11}$, or $t=2$, $n=6$, $q=3$ and $H=N \cong Z_{2}.M_{12}$  by \cite[Tables 8.2, 8.4, 8.9, 8.19. 8.25, 8.36 and 8.43]{BHRD}. However, $11$ does not divide the order $Z(N)$ in any of these groups and hence they are ruled out. Thus $M \trianglelefteq G(\Sigma)$. Moreover, $1 \neq M^{\Delta_{i}}\trianglelefteq G_{\Delta_{i}}^{\Delta_{i}}$ for each $\Delta_{i} \in \Sigma$, since $C \leq M$, and $M \cap G(\Delta_{i}) \leq Q$ for each $\Delta_{i} \in \Sigma$, since $M/Q \cong N/Z(N) \cong S$, with $S$ non-abelian simple. Hence $T \triangleleft M^{\Delta_{i}}$ and a quotient group of $M^{\Delta_{i}}$ is non-abelian simple. 

Since $M^{\Delta_{i}}$ is non-solvable and since $p^{m}+2=u^{h}$ with $u$ prime and $p^{m}>3$, by Theorem \ref{Monty}, $M^{\Delta_{i}}=T:(M^{\Delta_{i}})_{0}$ and one of the following holds:
\begin{enumerate}
\item $SL_{2}(p^{m}) \trianglelefteq (M^{\Delta_{i}})_{0}\leq \Gamma L_{2}(p^{m})$;
\item $SU_{3}(p^{m/3})\trianglelefteq (M^{\Delta_{i}})_{0}\leq (Z_{p^{m/3}-1}\times SU_{3}(p^{m/3})).Z_{2m/3}$, $m \equiv 0 \pmod{3}$;
\item $Sp_{4}(p^{m/2})%
\trianglelefteq (M^{\Delta_{i}})_{0}\leq \Gamma Sp_{4}(p^{m/2})$, $m$ even;
\item $(M^{\Delta_{i}})_{0}\leq (Q_{8} \circ D_{8}).S_{5}$ and $p^{m}=9$.
\end{enumerate}
Note that, in the previous list some automorphism groups of non-isomorphic $2$-designs listed in Theorem \ref{Monty} are brought together. Indeed, the group in (2.c.i) is a subgroup of $\Gamma L_{2}(p^{m})$ as well as that is (2.c.vi) of Theorem \ref{Monty} is a subgroup of the full translation complement of Hall plane of order, which is $(Q_{8} \circ D_{8}).S_{5}$ by \cite[Theorem II.8.3]{Lu}. Either $S \cong PSL_{2}(p^{m})$ or $S \cong PSU_{3}(p^{m/3})$ or $S \cong PSp_{4}(p^{m/2})$, or $S \cong PSL_{2}(5)$ for $p^{m}=9$, since $M^{\Delta_{i}} \cong M/(M \cap G(\Delta_{i}))$, $M \cap G(\Delta_{i}) \leq Q$ and $M/Q \cong N/Z(N) \cong S$, with $S$ non-abelian simple. However, both cases are ruled out since they violates $\left\vert G^{\Sigma }\right\vert \mid (q-1)\left\vert Out(S) \right \vert$, being $\left \vert G^{\Sigma }\right\vert$ divisible by $p^{m}+2$ with $p^{m}>3$. Thus $H$ lies in a member of $\mathcal{C}%
_{i}(\Gamma)$ for some $i$ such that $1\leq i\leq 7$.

The group $H$ does not lie in a member of $\mathcal{C}_{1}(\Gamma)$, since $H$ is
irreducible subgroup of $\Gamma$ by Proposition \ref{irre} and subsequent remark, and does not lie in a member of $\mathcal{C}_{3}(\Gamma)$ by the definition of $%
q $. Also, $H$ does not lies in a member of $\mathcal{C}_{5}(\Gamma)$. Indeed, if not so, then $n<4m$ since $q=p$ for $n=4m$. Then $H^{\ast}$ lies in a member of $\mathcal{C}_{5}(GL_{n}(q))$, but this is impossible by \cite[Theorem 3.1]{BP}, since $\Phi _{2m}^{\ast }(p)\mid \left\vert H^{\ast}\right\vert $ by Proposition \ref{primpart}. 
Assume that $H$ lies in a member of $\mathcal{C}_{2}(\Gamma)$. Then $H$ stabilizes a
decomposition of $V=V_{1}\oplus \cdots \oplus V_{n_{0}}$, $n_{0}>1$, where $\dim
V_{1}=\cdots =\dim V_{n_{0}}=m_{0}\geq 1$. Thus $n=m_{0}n_{0}$ and hence $%
H^{\ast}\leq GL_{m_{0}}(q)\wr S_{n_{0}}$.

If $n<4m$, then $m_{0}=1$, $n_{0}=n$, $q=p$, and either $p^{m}=3^{2}$ and $n \leq 7$, or $p^{m}=3^{3},5^{3}$,
and $n\leq 11$ by \cite[Theorem 3.1]{BP}, since $\Phi _{2m}^{\ast }(p)\mid \left\vert H^{\ast}\right\vert $ and since $p$ is odd. Then $u=11,29$
or $127$ respectively. Since $u$ does not divide the order of the corresponding $%
GL_{1}(q)\wr S_{n}$, then $u$ does not divide the order of $H^{\ast}$, whereas $u$ must divide it by Proposition \ref{primpart}. Indeed, in each of these cases $(u,(2m+t)/n)=1$ since $u$ a prime such that $u>4m \leq(2m+t)/n$. Hence, these cases are ruled out.

If $n=4m$, then $q=p$ and hence $H=H^{\ast}\leq GL_{m_{0}}(p)\wr S_{n_{0}}$. Let $w$
be a prime divisor of $p^{2m}-1$. Then $w$ divides either the order of $%
GL_{m_{0}}(p)$ or that of $S_{n_{0}}$. The former yields $%
2m\leq m_{0}$ by \cite[Proposition 5.2.15(i)]{KL}. Therefore $\left(
2m\right) n_{0}\leq m_{0}n_{0}=n= 4m$ and hence $n_{0}=2$ and $m_{0}=2m$, since $n_{0}>1$.

The case where $w$ divides the order of $S_{n_{0}}$ yields $w=2m+1$, since $%
w\equiv 1\pmod{2m}$ by \cite[Proposition 5.2.15(ii)]{KL}, and $n_{0}=4m$ and $m_{0}=1$, since $n_{0} \mid 4m$ and $n_{0}\geq w=2m+1>2$. Note that $n_{0}=2$ and $m_{0}=2m$ is clearly not compatible with $n_{0}=4m$ and $m_{0}=1$. Hence, in the latter case $\Phi_{2m}^{\ast}(p)=(2m+1)^{s}$ for some $s \geq 1$ and it is a divisor of $n_{0}!$. Then $s<\frac{n_{0}-1}{2m}< 4m/2m=2$ by \cite[Exercise 2.6.8]{DM} and hence $\Phi_{2m}^{\ast}(p)=2m+1$. Thus $p^{m}=3^{2},3^{3},5^{3}$ by \cite[Lemma 6.1.(i)]{BP}, and hence $p^{m}+2=11,29,127$ respectively. Then $p^{m}+2$ does not divide the order of the corresponding $H$, since $H \leq GL_{1}(p)\wr S_{4m}$, and hence these cases are ruled out by Lemma \ref{Phistar}(2). Therefore $n_{0}=2$, $m_{0}=2m$ and hence $H\leq GL_{2m}(p)\wr S_{2}$ preserves a decomposition $V=V_{1}\oplus V_{2}$, which is (1).

Assume that $H$ lies in a member of $\mathcal{C}_{4}(\Gamma)$. Then $H$ preserves a
tensor decomposition $V=V_{1}\otimes V_{2}$, with $\dim V_{i}=n_{i}$ and $%
1\leq n_{1}<n_{2}$. Therefore, $H^{\ast}\leq GL_{n_{1}}(q)\circ GL_{n_{2}}(q)$. No
cases arise for $n<4m$ by \cite[Theorem 3.1]{BP}, since $\Phi _{2m}^{\ast }(p)\mid \left\vert H^{\ast}\right\vert $. Thus $n=4m$ and $%
q=p$, hence $H=H^{\ast}\leq GL_{2}(p)\circ GL_{2m}(p)$ and we obtain (2).

Assume that $H$ lies in a member of $\mathcal{C}_{6}(\Gamma)$. Then $H$ lies in the
normalizer in $\Gamma L_{n}(q)$ of an absolutely irreducible symplectic type 
$s$-group $R$, with $s\neq p$. Hence $n=s^{y}$ for some $y \geq 1$ by \cite[Definition (c)
at p. 150.]{KL}

If $n<4m$, then $(q,n)=(3,4)$ or $(3,8)$ by \cite[Theorem 3.1]{BP}, since $\Phi _{2m}^{\ast }(p)\mid \left\vert H^{\ast}\right\vert $ and $%
u^{h}=q^{n/2}+2$ with $q$ odd. Thus $h=1$. Moreover, $n=2m+t$ since $q=p$,
and therefore $(m,t)=(2,2)$ for $(q,n)=(3,4)$, and $(m,t)=(3,2),(4,0)$ for $%
(q,n)=(3,8)$, since $n<4m$. A similar argument to that of the $\mathcal{S}$-case yields $%
u\mid 3^{2m+t-f}-1$, with $0<2m+t-f\leq 4m$, where $3^{f}$ is the number of fixed points of an element of order $u$ of $H$. Either $m=1$ and$\ 2+t-f=4$, or $m=2$ and $4+t-f=5$ by Lemma \ref{div}, whereas $m=2$ or $3,4$
respectively. Thus $H$ does not lie in a member of $\mathcal{C}_{6}(\Gamma)$ for $n<4m$.\\
If $n=4m$, then $q=p$ and hence $s=2$ and $y \geq 2$ as $n=s^{y}$. Therefore, $m=2^{y-2}$.
Let $w$ be a primitive prime divisor of $p^{2m}-1$. Then $w=2^{y-1}j+1=2mj+1$ for
some $j\geq 1$ by \cite[Proposition 5.2.15(ii)]{KL}. On the other hand, $w$ divides he order of $H$, where $H/(H \cap Z)$ is a subgroup of one of the groups given in \cite[Table 4.6.A]{KL} for $s=2$, and $Z$ is the center of $GL_{4m}(p)$. It follows that $w$ divides the order of $Sp_{2y}(2)$ and hence it divides either $2^{i}-1$ or $2^{i}+1$ for some $1 \leq i\leq y$. Then either $w=2^{y-1}+1=2m+1$ or $%
w=2^{y}+1=2(2m)+1$, respectively, and hence $\Phi _{2m}^{\ast }(p)$ is either $2m+1$
or $2(2m)+1$ or $\left( 2m+1\right) \left( 2(2m)+1\right) $. Then $%
p^{m}=3^{2},3^{3},5^{3},3^{9}$ or $17^{3}$ by \cite[Lemma 6.1]{BP}, since $%
q=p$ and $p$ is odd. Actually $p^{m}\neq 3^{9},17^{3}$, since they do not
fulfill $u^{h}=p^{m}+2$, whereas $h=1$ in the remaining cases. Then $u\mid
p^{4m-f}-1$, and arguing as above we obtain $p^{m}=3^{2}$, $u=11$ and $f=3$
and $n=8$. Since $q=p=3$, it follows that $Z(R)\cong Z_{2}$, $R\cong D_{8}\circ Q_{8}$ and hence $n=2$
by \cite[Definition (c) and Table 4.6.B at p. 150]{KL}. However, it contradicts $n=8$ and hence it is ruled out.

Assume that $H$ lies in a member of $\mathcal{C}_{7}(\Gamma)$. Then $H$ stabilizes a
decomposition of $V=V_{1}\otimes \cdots \otimes V_{n_{0}}$, where $\dim
V_{1}=\cdots =\dim V_{n_{0}}=m_{0}$. Hence $n=m_{0}^{n_{0}}$, with $%
m_{0}\geq 1$ and $n_{0}\geq 2$, and $H^{\ast}\leq (GL_{m_{0}}(p)\circ \cdots \circ
GL_{m_{0}}(p)).S_{n_{0}}$. If $m_{0}<3$, then $H$ lies in a member of $%
\mathcal{C}_{8}(\Gamma)$ (see remark before Proposition 4.7.3 in \cite{KL}),
which is not the case. Hence $m_{0}\geq 3$. No cases arise for $n<4m$ by 
\cite[Theorem 3.1]{BP}, since $\Phi _{2m}^{\ast }(p)\mid \left\vert H^{\ast}\right\vert $ by Proposition \ref{primpart}. Thus $n=4m$ and $q=p$ and hence $H=H^{\ast}$. A prime divisor $w$ of $%
p^{2m}-1$ divides either the order of $GL_{m_{0}}(p)$ or that of $S_{n_{0}}$%
. Since $w\equiv 1\pmod{2m}$ by \cite[Proposition 5.2.15(ii)]{KL}, the
latter implies $n_{0}\geq 2m+1$ and hence $3^{2m+1}\leq m_{0}^{2m+1}\leq 4m$%
, a contradiction. Thus $w$ divides the order of $GL_{m_{0}}(p)$. Then $%
2m\leq n_{0}$ by \cite[Proposition 5.2.15(i)]{KL}, and hence $%
(2m)^{n_{0}}\leq m_{0}^{n_{0}}=n=4m$. So $n_{0}=m_{0}=2$, and we reach a contradiction as $m_{0} \geq 3$. This completes the proof.
\end{proof}

\bigskip

\begin{lemma}
\label{dec1}The following hold in case (1) of Proposition \ref{dec}:

\begin{enumerate}
\item $H_{V_{j}}^{V_{j}}$ acts irreducibly on $V_{j}$ for $j=1,2$; 
\item $\Phi _{2m}^{\ast }(p)\frac{p^{m}+2}{(p^{m}+2,3)}\mid \left\vert
H_{V_{j}}^{V_{j}}\right\vert $ for $j=1,2$;
\item $H_{V_{j}}^{V_{j}}$ contains a normal subgroup $Q_{j}$ of order divisible by $\frac{p^{m}+2}{(p^{m}+2,3)}$ and such that $\Phi _{2m}^{\ast}(p) \mid \left[ H_{V_{j}}^{V_{j}}:Q_{j}\right]$ for $j=1,2$.
\end{enumerate}
\end{lemma}

\begin{proof}
The length of $V(\Delta_{i})^{G}$ is $p^{m}+2$ since $G$ acts transitively on $\Sigma$ and $V \trianglelefteq G$. Then the length of $V(\Delta_{i})^{H}$ is $p^{m}+2$, since $C=C_{G}(V)$. Let $w$ be a primitive prime divisor of $p^{2m}-1$, and since $G_{\Delta_{i}}^{\Delta_{i}}$ acts flag-transitively on $\mathcal{D}_{i}$, let $W_{i}$ be a Sylow $w$-subgroup of $G$ preserving $\Delta _{i}$. Then $W_{i}$ normalizes $V(\Delta _{i})$. Moreover, $W_{i}$ acts faithfully on $V$ inducing a Sylow $w $-subgroup of $H$, since $C$ is a $p$-group by Proposition \ref{qcDv}(3) and Theorem \ref{centraliz}.\\ Assume that $%
W_{i}(V_{1})\neq 1$. Then $W_{i}(V_{1})$ acts irreducibly on $V_{2}$. If it is not so, there is $\zeta \in W_{i}(V_{1})$, $\zeta \neq 1$, fixing $V_{2}$ pointwise, since $\dim V_{2}=2m$ and $\zeta$ is a $w$-element with $w$ a primitive prime divisor of $p^{2m}-1$. Then $\zeta
\in C$ and hence $\zeta$ is a $p$-element, since $C$ is a $p$-group, whereas $\zeta$ is a non-trivial $w$-element with $w\neq p$. Clearly $W_{i}(V_{1})$ preserves $\left\langle
x_{1}\right\rangle \oplus V_{2}$ for each $x_{1}\in V_{1}$, $x_{1}\neq 0$. Since $\dim \left\langle x_{1}\right\rangle \oplus
V_{2}=2m+1$ and $\dim V(\Delta_{i})=2m$, it follows that $V(\Delta _{i})\cap \left( \left\langle
x_{1}\right\rangle \oplus V_{2}\right) \neq 0$. Let $\lambda \in GF(p)$ and $x_{2} \in V_{2}$ such that $\lambda x_{1}+ x_{2} \in V(\Delta _{i})\cap \left( \left\langle x_{1}\right\rangle \oplus V_{2}\right)$. If $x_{2} \neq 0$, there is a non-trivial $\alpha \in W_{i}(V_{1})$ such that $x_{2}^{\alpha}\neq x_{2}$. Then $\lambda x_{1}+ x_{2}^{\alpha} \in V(\Delta _{i})\cap \left( \left\langle x_{1}\right\rangle \oplus V_{2}\right)$, since $W_{i}(V_{1})$, and hence $\alpha$, preserves $V(\Delta _{i})\cap \left( \left\langle x_{1}\right\rangle \oplus V_{2}\right)$. Hence $x_{2}^{\alpha}- x_{2}$ is a non-zero element of $V(\Delta _{i})\cap V_{2}$, since 
$x_{2}^{\alpha}- x_{2}=(\lambda x_{1}+ x_{2}^{\alpha})-(\lambda x_{1}+ x_{2})$ and $x_{2}^{\alpha} \neq x_{2}$. Thus $V(\Delta _{i})=V_{2}$ since $W_{i}(V_{1})$ acts irreducibly on $V_{2}$ and preserves $V(\Delta _{i})$. So $p^{m}+2=2$, since the length of $V(\Delta_{i})^{H}$ is $p^{m}+2$ and since $H$ switches $V_{1}$ and $V_{2}$, and we reach a contradiction. Thus $x_{2} = 0$ and hence $\lambda x_{1} \in V(\Delta_{i}) \cap V_{1}$. Then $V(\Delta_{i})=V_{1}$ and hence $\left \vert V(\Delta_{i})^{H} \right \vert = p^{m}+2=2$, which is a contradiction. Thus $W_{i}$ acts
faithfully and irreducibly on $V_{1}$. Similarly, we prove that $W_{i}$ acts
faithfully and irreducibly on $V_{2}$. Thus $H_{V_{j}}^{V_{j}}$ acts irreducibly on $V_{j}$, which is (1), and $\Phi _{2m}^{\ast }(p)\mid
\left\vert H_{V_{j}}^{V_{j}}\right\vert$.   

Recall that $U$ is the pre-image of $Soc(G^{\Sigma})$. Let $S$ be a Sylow $u$-subgroup of $U$, where $u^{h}=p^{m}+2$. Then $S$ acts faithfully on $V$ inducing a Sylow $u $-subgroup of $H$, since $C$ is a $p$-group. Assume that $S(V_{1})\neq 1$.
Since $(p^{m}+2,p^{2m}-1)\mid 3$, there is a subgroup of $Y$ of $S(V_{1})$
such that $\left[ S(V_{1}):Y\right] \leq (3,u)$ fixing an non-zero element $%
z_{2}$ of $V_{2}$. Then $Y$ fixes $V_{1}\oplus \left\langle
z_{2}\right\rangle $ pointwise. Since $V(\Delta _{i})\cap \left( V_{1}\oplus
\left\langle z_{2}\right\rangle \right) \neq 0$ for each $i=1,...,p^{m}$, there is an element of $%
v_{i}\in V(\Delta _{i})$, $v_{i}\neq 0$, fixed by $Y$ for each $i=1,...,p^{m}$.

Assume that $Y\nleq G(\Sigma )$. Then there are $\eta \in Y$ and $i_{0}\in
\left\{ 1,...,p^{m}+2\right\} $ such that $\left[ v_{i_{0}},\eta \right] =1$
and $\Delta _{i_{0}}^{\eta }\neq \Delta _{i_{0}}$. Then $v_{i_{0}}\in
V\left( \Delta _{i_{0}}\right) \cap V\left( \Delta _{i_{0}}^{\eta }\right) $%
, with $v_{i}\neq 0$, whereas $V\left( \Delta _{i_{0}}\right) \cap V\left(
\Delta _{i_{0}}^{\eta }\right) =0$ by Lemma \ref{if}(2), since $\Delta _{i_{0}}^{\eta
}\neq \Delta _{i_{0}}$. Thus $Y\leq G(\Sigma )$ and hence $\left[
S(V_{1}):S(V_{1})\cap G(\Sigma )\right] \leq (3,u)$. Similarly, we have $%
\left[ S(V_{2}):S(V_{2})\cap G(\Sigma )\right] \leq (3,u)$. Thus 
\[
\left\vert S^{V_{j}}\right\vert =\frac{\left\vert S\right\vert }{\left\vert
S(V_{j})\right\vert }\geq \frac{\left\vert S\right\vert }{(3,u)\left\vert
S(V_{2})\cap G(\Sigma )\right\vert }\geq \frac{\left\vert S\right\vert }{%
(3,u)\left\vert S\cap G(\Sigma )\right\vert }= \frac{\left\vert U^{\Sigma
}\right\vert }{(3,u)}=\frac{p^{m}+2}{(p^{m}+2,3)}\text{,} 
\]%
since $U=G(\Sigma )S$. Thus $\frac{p^{m}+2}{(p^{m}+2,3)}\mid \left\vert
S^{V_{j}}\right\vert $, since $u^{h}=p^{m}+2$ and $S^{V_{j}}$ is a $u$-group, and hence $\Phi _{2m}^{\ast }(p)\frac{p^{m}+2%
}{(p^{m}+2,3)}\mid \left\vert H_{V_{j}}^{V_{j}}\right\vert $, since we have already proven that $\Phi _{2m}^{\ast }(p)\mid
\left\vert H_{V_{j}}^{V_{j}}\right\vert $. Therefore, we get (2). Moreover $S^{V_{j}}\leq Q_{j}\trianglelefteq H^{V_{j}}$, where $Q_{j}=(U/C)^{V_{j}}$, since $C\trianglelefteq G(\Sigma)\trianglelefteq U$ by Theorem \ref{centraliz}. Hence the order of $Q_{j}$ is divisible by $\frac{p^{m}+2}{(p^{m}+2,3)}$. Also, since $G^{\Sigma}$ acts $2$-transitively on $\Sigma$, it follows that $\Phi _{2m}^{\ast }(p)$ divides $[G^{\Sigma}:U^{\Sigma}]$ and hence $[G:U]$. Then $\Phi _{2m}^{\ast }(p)$ divides $[H:U/C]$, since $C\trianglelefteq U$ and $C$ is a $p$-group, and hence $[H_{V_{j}}^{V_{j}}:Q_{j}]$ for each $j=1,2$, which is (3).
\end{proof}

\begin{lemma}
\label{dec2}The following hold in case (2) of Proposition \ref{dec}:

\begin{enumerate}

\item $H^{V_{2}}$ acts irreducibly on $V_{2}$;

\item $\Phi _{2m}^{\ast }(p)\frac{p^{m}+2}{(p^{m}+2,3)}\mid \left\vert
H^{V_{2}}\right\vert $;
\item $H^{V_{2}}$ contains a normal subgroup $Q_{2}$ of order divisible by $\frac{p^{m}+2}{(p^{m}+2,3)}$ and such that $\Phi _{2m}^{\ast}(p) \mid \left[ H^{V_{2}}:Q_{2}\right]$.
\end{enumerate}
\end{lemma}

\begin{proof}
Set $M=GL_{2}(p)\circ GL_{2m}(p)$. Let $w$ be primitive prime divisor of $%
p^{2m}-1$ and let $W$ be a Sylow $w$-subgroup of $H$. Then $W^{V_{1}}\leq
H^{V_{1}}$ and hence $W^{V_{1}}=1$, since $m>1$ by Proposition \ref{dec}(2). Therefore, $W^{V_{2}}\cong W$,
since $V=V_{1}\otimes V_{2}$ and $W\cap Z(M)=1$, being $Z(M)\leq Z_{p-1}$.
Thus $H^{V_{2}}$ acts irreducibly on $V_{2}$, which is (1), and $\Phi _{2m}^{\ast }(p)\mid \left\vert H^{V_{2}}\right\vert $.

Let $S$ be a Sylow $u$-subgroup of $U$, where $U\cong Soc(G^{\Sigma})$. Then $S$ acts faithfully on $V$ inducing a Sylow $u $-subgroup of $H$, where $u^{h}=p^{m}+2$ and $p$ is odd, since $C$ is a $p$-group by Proposition \ref{qcDv}(3) and Theorem \ref{centraliz}. Note that $S^{V_{1}}$ is a subgroup of $ GL_{2}(p)$ of order at most $9$, since $(p^{m}+2,(p-1)^{2})\mid 9$. Hence $S=S(V_{1})$ or $\left[ S:S(V_{1})\right] \leq 9$ according to whether $3$ does not divide or
does divide the order of $Z(M)$. Thus $\left\vert S^{V_{2}}\right\vert \geq \frac{%
\left\vert S\right\vert }{(p^{m}+2,3)}$ since $S(V_{1})\cap S(V_{2})\leq
Z(M)\leq Z_{p-1}$. Then $\frac{p^{m}+2}{(p^{m}+2,3)}\mid \left\vert
S^{V_{2}}\right\vert $, since $u^{h}=p^{m}+2$ and $S^{V_{2}}$ is a $u$-group, hence $\Phi _{2m}^{\ast }(p)\frac{p^{m}+2}{%
(p^{m}+2,3)}\mid \left\vert H^{V_{2}}\right\vert $, which is (2).\\ Since $C\trianglelefteq G(\Sigma)\trianglelefteq U$ by Theorem \ref{centraliz}, let $Q_{2}=(U/C)^{V_{2}}$. Then the order of $Q_{2}$ is divisible by $\frac{p^{m}+2}{(p^{m}+2,3)}$, since $S^{V_{2}}\leq Q_{2}\trianglelefteq H^{V_{2}}$. Now a similar argument to that of the proof of Lemma \ref{dec1}(3) can be used to obtain that $\Phi _{2m}^{\ast }(p)$ divides $[H^{V_{2}}:Q_{2}]$. Thus we obtain (3).
\end{proof}

\bigskip

\begin{remark}
\label{rem}In view of Lemmas \ref{dec1} and \ref{dec2}, in (1) and in (2) of Proposition \ref{dec} there is a quotient group $X$ of a subgroup of $H$ with the following properties:
\begin{enumerate}
\item $X$ is an irreducible subgroup of $GL(V_{2})$ of order divisible by $\Phi _{2m}^{\ast }(p)\frac{p^{m}+2}{(p^{m}+2,3)}$; 
\item $X$ contains a normal subgroup $Q$ of order divisible by $\frac{p^{m}+2}{(p^{m}+2,3)}$  and such that $\Phi _{2m}^{\ast}(p) \mid \left[ X:Q\right]$;
\item $p^{m}+2=u^{h}$, where $u$ is a prime and $h \geq 1$.
\end{enumerate}
We are going to show that a quotient group of $H$ with such constraints does not exist. We derive from this fact that $\Phi _{2m}^{\ast }(p)=1$, hence $(p^{m},u^{h})= (3,5),(7,9)$ by Lemma \ref{Phistar}(3).
\end{remark}

\bigskip

For each divisor $\ell$ of $2m$ the group $\Gamma L_{\ell}(p^{2m/\ell})$ has a
natural irreducible action on $V_{2}$. By Proposition \ref{irre} we may choose $\ell$ to be minimal such that 
$X\leq \Gamma L_{\ell}(p^{2m/\ell})$ in this action and write $%
a=p^{2m/\ell}$.

\bigskip

\begin{lemma}
\label{Uredu} The following hold
\begin{enumerate}
\item $\ell>1$;
\item $Q \nleq Z(GL_{\ell}(a))$.
\end{enumerate}
\end{lemma}

\begin{proof}
Suppose that $\ell=1$. Then $X\leq \Gamma L_{1}(a)$ and hence $%
\Phi _{2m}^{\ast }(p)\frac{p^{m}+2}{(p^{m}+2,3)}\mid (p^{2m}-1)\cdot (2m)$. Then we obtain (1) by proceeding as in Lemma \ref{n>1} (for $t=0$) with $\frac{p^{m}+2}{(p^{m}+2,3)}$ in the role of $p^{m}+2$ and bearing in mind that $p^{m}\neq 7$.
 
Suppose the contrary. Then $\frac{p^{m}+2}{(p^{m}+2,3)} \mid p^{2m/\ell}-1$ and hence $p^{m}+2 \mid 9$. So $(p^{m},u^{h})=(7,9)$, since $p^{m} \geq 3$ by Lemma \ref{c=d}, but his case is ruled out since it contradicts the assumption $\Phi _{2m}^{\ast }(p)>1$. Thus we obtain (2).
\end{proof}

\bigskip

Let $X^{\ast }=X\cap GL_{\ell}(a)$. Then $\Phi _{2m}^{\ast }(p)$ divides the order of $X^{\ast}$, since $(\Phi _{2m}^{\ast }(p),\ell)=1$ by \cite[Proposition 5.2.15.(ii)]{KL}, being $\ell \mid 2m$. Then one of the following holds by \cite[Theorem 3.1]{BP} and by the minimality of $\ell$:

\begin{enumerate}
\item[(i).] $X^{\ast }$ contains as a normal subgroup $Y$ isomorphic to one of the
groups $SL_{\ell}(a)$ with $\ell \geq2$, $Sp_{\ell}(a)$ or $\Omega _{\ell}^{-}(a)$ with $\ell$ even and $ \ell \geq 2$, or $%
SU_{\ell}(a^{1/2})\trianglelefteq X$ with $a$ square, $\ell$ odd and $\ell\geq 3$.
\item[(ii).] $X^{\ast }\leq (D_{8}\circ Q_{8}).S_{5}$ and $(\ell,a)=(4,3)$.

\item[(iii).] $X^{\ast }$ is \emph{nearly simple}, that is, $S\trianglelefteq
X^{\ast }/\left( X^{\ast }\cap Z\right) \leq Aut(S)$, where $Z$ is the center of $GL_{\ell}(a)$
and $S$ is a non-abelian simple group. Moreover, if $Y$ is the pre-image of $S$ in $X^{\ast }$, then $Y$ is absolutely irreducible on $V$ and $Y$ is not a
classical group defined over a subfield of $GF(a)$ in its natural
representation.
\end{enumerate}

\begin{theorem}\label{sevdah}
$(p^{m},u^{h})= (3,5),(7,9)$.
\end{theorem}

\begin{proof}
Assume that (i) occurs. Then $X \leq (Z \circ X^{\ast}).Out(\bar{X^{\ast}})$, with $\bar{X^{\ast}}= X^{\ast}/(X^{\ast} \cap Z)$, where $Z=Z(GL_{\ell}(a))$. It follows that $X^{\ast} \trianglelefteq Q $ by Lemma \ref{Uredu}(2), since $Q \trianglelefteq X$. Then $\Phi _{2m}^{\ast }(p) \mid (a-1)\left \vert Out(X^{\ast}) \right \vert$, since $\Phi _{2m}^{\ast}(p) \mid \left[ X:Q\right]$ by Remark \ref{rem}(3), and hence $\Phi _{2m}^{\ast}(p) \mid (p^{2m/\ell}-1)2m$. However this impossible by \cite[Proposition 5.2.15]{KL}, since $\ell >1$ by Lemma \ref{Uredu}(1). 

Case (ii) is ruled out, since $a^{\ell/2}+2=11$ does not divide the order of $X$

Assume that case (iii) occurs. Suppose that $S\cong A_{s }$, $s \geq 5$ and that $V_{2}$ is the fully
deleted permutation module for $A_{s }$. Then $a=p$, $n=2m$, $A_{s }\trianglelefteq X^{\ast} \leq S_{s}\times Z$, where $Z$ is the center of $GL_{2m}(p)$, and either $s=2m+1$ or $s=2m+2$ according to whether $p$ does not divide or does divide $s $, respectively, by \cite[Lemma 5.3.4]{KL}.
Moreover $p^{m}=3^{2},3^{3},5^{3}$ by \cite[Theorem 3.1]{BP}, since $p$ is odd, and hence $p^{m}+2=11,29,127$ respectively. However, such values of $p^{m}+2$ do not divide $\left\vert X\right\vert $ and hence they are ruled out (see Remark \ref{rem}).

Assume that $S\cong A_{s}$, $s \geq 5$ and that $V_{2}$ is not the fully deleted permutation module for $A_{s}$. Then $(\ell,a)=(4,7)$, $Z_{2}.A_{7}\trianglelefteq H\leq
Z_{2}.S_{7}\times Z_{3}$ and $V_{2}=V_{4}(7)$ by \cite[Theorem 3.1]{BP}, since $p$ is odd. However it is ruled out, since $7^{2}+2$ does not divide the order of $X$.

Assume that $S$ is sporadic. Then $S\cong J_{2}$ and $(\ell,a)=(6,5)$ by \cite[Theorem 3.1]{BP}, since $p$ is odd. However $a^{\ell/2}+2=127$ does not divide the order of $X$ and hence this case is excluded.

Assume that $S$ is a Lie type simple group in characteristic $%
p^{\prime }$. Then $S$ is given in \cite[Theorem 3.1]{BP} and recorded in Table \ref{T1}. Since $\frac{a^{\ell/2}+2}{(a^{\ell/2}+2,3)}\mid \left\vert X\right\vert $ and since $%
S\trianglelefteq X^{\ast}/\left( X^{\ast}\cap Z\right) \leq Aut(S)$, it follows that $$%
\varrho =\frac{a^{\ell/2}+2}{\left( a^{\ell/2}+2,3(a-1)\left\vert Aut(S)\right\vert
\right) }$$ must divide the order of $S$. However, the order of $S$ is divisible by the corresponding $\varrho $ in none of the cases listed in Table \ref{T1}. Hence all the groups in Table \ref{T1} are excluded.
\begin{table}\label{T1}
\caption{\small{Admissible $S$ and corresponding $\varrho$}}
\centering
\normalsize
\begin{tabular}{|l|l|l|l|}
\hline
$S$ & $\ell$ & $a$ & $\varrho $ \\ \hline
$PSL_{2}(7)$ & $6$ & $3$ & $29$ \\ \hline 
              & $6$ & $5$ & $127$ \\ \hline
              & $3$ & $9$ & $29$  \\ \hline
              & $3$ & $5^{2}$ & $127$  \\ \hline
$PSL_{2}(13)$ & $6$ & $3$ & $29$ \\ \hline
$PSL_{3}(2^{2})$ & $6$ & $3$ & $29$ \\ \hline
$PSU_{3}(3)$ & $6$ & $5$ & $127$ \\ \hline
\end{tabular}%
\end{table}

Assume that $S$ is a Lie type simple group in characteristic $p$. Then $p=2$ by \cite[Theorem 3.1]{BP}, whereas $p$ must be odd, and hence no cases arise. Thus $\Phi _{2m}^{\ast }(p)=1$ and hence $(p^{m},u^{h})= (3,5),(7,9)$ by Lemma \ref{Phistar}(3).
\end{proof}

\begin{theorem}\label{pupicchia}
If $G^{\Sigma}$ is of affine type, then  $\mathcal{D}$ is isomorphic to the $2$-$(45,12,3)$ design constructed in \cite[Construction 4.2]{P}. 
\end{theorem}

\begin{proof}[Proof of Theorem \ref{t1}]
It follows from Theorem \ref{sevdah} that $(p^{m},u^{h})=(3,5)$ or $(7,9)$, and the assertion follows in the former case by \cite[Corollary 4.2]{P}. Hence, in order to complete the proof we need to rule out $(p^{m},u^{h})=(7,9)$. We are going to prove this in a series of steps.  
\begin{enumerate}
\item[(1).] \textbf{Let $\Delta_{i} \in \Sigma$ and $x_{i} \in \Delta_{i}$. Then $G^{\Sigma }\cong AGL_{1}(9)$, $G_{x_{i}}^{\Delta
_{i}}\cong Z_{3^{j}}\times
Z_{16}$ and $G(\Sigma )_{x_{i}}^{\Delta _{i}}\cong Z_{3^{j}}\times Z_{2}$
where $j=0,1$.}
\end{enumerate}

$G_{\Delta _{i}}^{\Sigma }\cong Z_{8},Q_{8},SD_{16},SL_{2}(3),GL_{2}(3)$ by \cite[Section 2, (B)]{Ka85}, since $G^{\Sigma}$ is an affine group acting $2$-transitively on $\Sigma$ by Proposition \ref{boom}. On the other hand, since $\mathcal{%
D}_{i}\cong AG_{2}(7)$ by Theorem \ref{Monty}, since $m=1$, it follows from \cite[Theorem 1' and Table II]{F1}
that either $SL_{2}(7)\trianglelefteq G_{x_{i}}^{\Delta
_{i}}\leq GL_{2}(7)$, or $ G_{x_{i}}^{\Delta
_{i}}\leq \Gamma L_{1}(7^{2})$ or $ G_{x_{i}}^{\Delta
_{i}}\cong SL_{2}(3).Z_{2}$. Moreover, if $G_{x_{i}}^{\Delta
_{i}}\leq \Gamma L_{1}(7^{2})$, it is not
difficult to see that $G_{x_{i}}^{\Delta
_{i}}\cong Z_{3^{j}}\times Z_{16},Z_{3^{j}}\times Q_{16},Z_{3^{j}}\times
SD_{32}$, with $j=0,1$. Using (\ref{isom}) in Corollary \ref{glue}, with $G_{\Delta
_{i}}^{\Sigma }\cong Z_{8},Q_{8},SD_{16},SL_{2}(3),GL_{2}(3)$, we see that
the unique possibilities are $G_{\Delta _{i}}^{\Sigma }\cong Z_{8}$ and
either $G_{x_{i}}^{\Delta
_{i}}\cong Z_{16}$
and $G(\Sigma )_{x_{i}}^{\Delta _{i}}\cong Z_{2}$, or $G_{x_{i}}^{\Delta
_{i}}\cong Z_{48}$ and $G(\Sigma
)_{x_{i}}^{\Delta _{i}}\cong Z_{8}$. Thus $G^{\Sigma }\cong AGL_{1}(9)$, $G_{x_{i}}^{\Delta
_{i}}\cong Z_{3^{j}}\times
Z_{16}$ and $G(\Sigma )_{x_{i}}^{\Delta _{i}}\cong Z_{3^{j}}\times Z_{2}$
with $j=0,1$.
\begin{enumerate}
\item[(2).] \textbf{$G(\Delta _{i})\leq E_{7^{2}}:(Z_{3^{j}}\times Z_{2})$ for each $i=1,...,9$.}
\end{enumerate}
Since $G(\Delta
_{i})\trianglelefteq G(\Sigma )$ and $G(\Delta _{i})\cap G(\Delta _{s})=1$ for each $s\in \left\{
1,...,9\right\} $, with $s\neq i$, by Lemma \ref{if}(2), it follows that $G(\Delta _{i})$ is
isomorphic to a normal subgroup of $G(\Sigma )^{\Delta _{s}}$. On the other
hand $G(\Sigma )^{\Delta _{s}}\cong E_{7^{2}}:Z_{3^{j}}\times Z_{2}$ with $%
j=0,1$, since $G(\Sigma )=G(\Sigma )_{x_{s}}V$ by Proposition \ref{qcDv}(1)(3) and since $G(\Sigma )_{x_{s}}^{\Delta _{s}}\cong Z_{3^{j}}\times Z_{2}$ with $j=0,1$ by (1). Thus $G(\Delta _{i})\leq E_{7^{2}}:Z_{3^{j}}\times Z_{2}$.
\begin{enumerate}
\item[(3).] \textbf{$G$ is solvable.}
\end{enumerate}
It follows from (2) that $G(\Delta_{i})$ is solvable. Thus $G(\Sigma )_{x_{i}}$ is solvable, since $G(\Sigma )_{x_{i}}^{\Delta _{i}}\cong Z_{3^{j}}\times Z_{2}$ by (1), and hence $G(\Sigma)$ is solvable since $G(\Sigma )=G(\Sigma )_{x_{i}}V$. Therefore $G$ is solvable, since $G^{\Sigma }\cong AGL_{1}(9)$.
\begin{enumerate}
\item[(4).] \textbf{$H$ is a solvable irreducible subgroup of $GL_{4}(7)$.}
\end{enumerate}
Recall that $H$ is an irreducible subgroup of $GL_{2+t}(7)$, with $t\leq 2$, by Proposition \ref{irre}, and let $R$ be the pre-image of $H\cap SL_{2+t}(7)$ in $G$. Then $%
R\vartriangleleft G$. Therefore $R^{\Sigma }\vartriangleleft G^{\Sigma }$
and hence either $R^{\Sigma }=1$ or $E_{9} \trianglelefteq R^{\Sigma }$, since $G^{\Sigma }\cong AGL_{1}(9)$. The former implies $R\leq G(\Sigma )$ and hence $9\mid \left\vert G/R\right\vert 
$. Then $9$ divides the index of  $%
SL_{2+t}(7)$ in $GL_{2+t}(7)$, since $G/R \cong H/(H\cap SL_{2+t}(7))$, which is a contradiction. Thus $E_{9} \trianglelefteq R^{\Sigma }$ and hence a quotient group of $H\cap SL_{2+t}(7)$ contains a normal subgroup isomorphic to $E_{9}$, since $C$ is a $7$-group by Proposition \ref{qcDv}(2) and Theorem \ref{centraliz}.\\
Let $M$ be the pre-image of $H\cap Z(GL_{2+t}(7))$ in $G$. Clearly $%
M\vartriangleleft G$. Therefore $M^{\Sigma }\vartriangleleft G^{\Sigma }$
and hence either $M^{\Sigma }=1$ or $9\mid \left\vert M^{\Sigma }\right\vert 
$, since $G^{\Sigma }\cong AGL_{1}(9)$. The latter implies $9\mid \left\vert
M/C\right\vert $, since $C$ is a $7$-group, whereas $%
M/C=H\cap Z(GL_{2+t}(V))\leq Z_{6}$. Thus $M^{\Sigma }=1$ and hence $M\leq
G(\Sigma )$.

Set $A=(H\cap SL_{2+t}(7))/(H\cap Z(SL_{2+t}(7)))$, then $A$ is isomorphic to a solvable
subgroup of $PSL_{2+t}(7)$. Moreover, a quotient group of $A$ contains $E_{9}$ as a normal subgroup, since $A \cong R/M$, $E_{9} \trianglelefteq R^{\Sigma }$ and $M\leq G(\Sigma)$. Thus $t \neq 0$.

Assume that $t=1$. Since $Z_{16} \leq G_{x_{i}}^{\Delta_{i}}$, then $H$ contains $2$-elements of order at least $16$ and hence $A$ contains elements of order at least $8$. Hence $A$ is isomorphic to a solvable subgroup of $PSL_{3}(7)$ of order divisible by $72$. Then $A \cong E_{9}:Q_{8}$ by \cite{At}, whereas $A$ contains elements of order at least $8$. Thus $t \neq 1$ and hence the claim follows from (3) and from Proposition \ref{irre}.
\begin{enumerate}
\item[(5).] \textbf{$G=C:(Q:J)$, where $\left\vert Q \right \vert=3^{2+j}$, with $0\leq j \leq 1$, and $J \cong Z_{16}$.}
\end{enumerate}
Let $S$ be a Sylow $7$-subgroup of $G$ containing $C$. Since $G^{\Sigma
}\cong AGL_{1}(9)$, it follows that $S\leq G(\Sigma )$. Since $G_{\Delta
_{i}}^{\Delta _{i}}\cong E_{7^{2}}:(Z_{3^{j}} \times Z_{16})$ and $G(\Delta _{i})\leq
E_{7^{2}}:(Z_{3^{j}} \times Z_{2})$ by (1) and (2), it follows that $\left\vert S\right\vert \leq 7^{4}$. On the other hand, $ 7^{4} \leq \left\vert C\right\vert \leq \left\vert S\right\vert$ by (4) and by Proposition \ref{qcDv}(3). Hence $S=C \trianglelefteq G$. Moreover, $\left \vert H \right \vert = 2^{4+e}\cdot 3^{2+f}$, where $e \leq 1$ and $j\leq f \leq 2$, again by (1) and (2). Then $G=C:K$, where $K$ is a group of order $2^{4+e}\cdot 3^{2+f}$ by \cite[Theorem 6.2.1(i)]{Go}.\\
Since $K^{\Sigma} \cong AGL_{1}(9)$, it results $\left \vert K(\Sigma) \right \vert =2^{1+e}\cdot 3^{f}$, with $e \leq 1$ and $j\leq f \leq 2$. Let $S$ be a Sylow $w$-subgroup of $K(\Sigma)$, where $w \in \{2,3\}$. If either $e=1$ or $f=2$, assume that $w$ is $2$ or $3$ respectively. Then $Z_{w} \leq S(\Delta_{i})$ for each $i \in \{1,...,9\}$, since $\left \vert S \right\vert =w^{2}$, and since $G(\Sigma )_{x_{i}}^{\Delta _{i}}\cong Z_{3^{j}}\times Z_{2}$, with $j=0,1$, for each $i \in \{1,...,9\}$ by (1). Then $S(\Delta_{i}) \cap S(\Delta_{i'}) \neq 1$ for some $i,i' \in \{1,...,9\}$, with $i \neq i'$, since the number of cyclic subgroup of $Q(\Sigma)$ is $w+1$, which is at most $4$,  and $\left \vert \Sigma \right \vert =9$. Since it contradicts Lemma \ref{if}(2), this case is excluded. Thus $e=0$, $f \leq 1$ and hence $\left \vert K(\Sigma) \right \vert =2\cdot 3^{f}$. From this fact and from $K^{\Sigma} \cong AGL_{1}(9)$, it results that the pre-image $P$ in $K$ of $Z_{3} \times Z_{3}$ is $Q:Z_{2}$, where $Q$ a Sylow $3$-subgroup of $K$. Moreover, the Frattini's argument implies $K=N_{K}(Q)P$ and hence $K=Q:J$, where $J$ is a Sylow $2$-subgroup of $K$. Finally, $J \cong Z_{16}$, since $J$ is of order $16$ and since $G_{x_{i}}^{\Delta_{i}}\cong Z_{3^{j}}\times Z_{16}$ by (1).   
\begin{enumerate}
\item[(6).] \textbf{$Q$ is abelian.}
\end{enumerate}
Suppose the contrary. Then $j=1$ and hence $Q$ is extraspecial. If there is an element $\phi$ in $Q$ of order $9$. Then $Fix(\phi ^{2})\neq 1$, since $V=V_{4}(7)$. Then $K$ preserves $Fix(\phi ^{2})$, since $%
\left\langle \phi ^{2}\right\rangle =Z(Q)$, whereas $K$ acts irreducibly
on $V$ by (4), since $K=G/C=H$. Thus $Q$ is of exponent $3$. Now $Z(Q) \leq K(\Sigma)$, since $K^{\Sigma} \cong AGL_{1}(9)$, therefore $Z(Q)$  preserves each $\Delta_{i}$ in $\Sigma$ and hence it normalizes $V(\Delta_{i})$. Then $Z(Q)$ is a reducible subgroup of $GL_{4}(7)$. Then $Q$ is reducible by \cite[Theorem 3.4.1]{Go}. Then either $V=X_{1}\oplus X_{2}\oplus X_{3}\oplus X_{4}$, where $X_{s}$, $s=1,2,3,4$, is a $Q$-invariant $1$-dimensional subspace of $V$, and $K\leq GL_{1}(7)\wr
S_{4}$, or $V=Y_{1}\oplus Y_{2}$, where $Y_{1}, Y_{2}$ are $Q$-invariant $2$-dimensional subspaces of $V$, and $K\leq GL_{2}(7)\wr Z_{2}$. In each case there is a $Q$-invariant subspace of $V$ fixed pointwise by a non trivial normal subgroup of $Q$, since the order of $Q$ is $3^{3}$. Also such a group contains $Z(Q)$. So $Fix(Z(Q))$ is a $H$-invariant subspace of $V$ of dimension at least $1$, since $Z(Q) \trianglelefteq K$, and we reach a contradiction since $K$ acts irreducibly on $V$. Thus $Q$ is abelian. 
\begin{enumerate}  
\item[(7).] \textbf{The final contradiction.}
\end{enumerate}
$Q$ acts reducibly on $V$ by \cite[Theorem 3.2.3]{Go}, since $Z_{3} \times Z_{3} \leq Q$, hence $K$ acts
transitively on a $Q$-invariant decomposition of $V$ in subspaces of equal
dimension by \cite[Theorem 3.4.1]{Go}, since $K$ acts irreducibly on $V$. Therefore either $K\leq GL_{1}(7)\wr
S_{4}$ or $K\leq GL_{2}(7)\wr Z_{2}$. However the former does not contain
cyclic subgroups of order $16$. Hence, $K\leq GL_{2}(7)\wr Z_{2}$ and $%
Q=\left\langle \alpha \right\rangle \times \left\langle \beta \right\rangle
\times \left\langle \gamma \right\rangle $, where $o(\alpha )=3^{j}$ and $%
o(\beta )=o(\gamma )=3$. Also $\alpha \in Z(K)$, whereas $N_{K}(\left\langle
\delta \right\rangle )\cong Z_{3}:Z_{4}$ for each $\delta \in Q \setminus
\left\langle \alpha \right\rangle $.\\
Let $V=V_{1}\oplus V_{2}$ be the decomposition preserved by $K$. Clearly $J_{V_{1}} \cong Z_{8}$ acts faithfully on $V_{1}$, since $J \cong Z_{16}$ switches $V_{1}$ and $V_{2}$. Also $Q$ preserves $V_{1}$ and $Q(V_{1}) \neq 1$. If $Q(V_{1})$ is of order $9$, then $Q(V_{1})\cap Q(V_{2}) \neq 1$ since $Q \trianglelefteq K$, the order of $Q$ is $3^{2+f}$, with $f \leq 1$, and since $K$ switches $V_{1}$ and $V_{2}$. However, this is impossible since it contradicts Lemma \ref{if}(2). Thus $Q(V_{1})$ is of order $3$ and hence $(Z_{3} \times Z_{3}):Z_{8} \leq K_{V_{1}}^{V_{1}} \leq GL_{2}(7)$, which is also impossible. So this case is ruled out and the proof is completed.
\end{proof}

\bigskip

\subsection{$G^{\Sigma}$ is an almost simple group} In this subsection we assume that $Soc(G^{\Sigma})$ is a non-abelian simple group.
\begin{proposition}
\label{Ao} One of the following holds:
\begin{enumerate}
\item[I.] $p^{m}=4$ and $\mathcal{D}$ is isomorphic to one of the four $2$-$(96,20,4)$ designs constructed in \cite{LPR}.
\item[II.] $p^{m}=9$ and $G^{\Sigma } \cong PSL_{2}(11),M_{11}$.
\item[III.] $p^{m}=16$ and $G^{\Sigma } \cong PSL_{2}(17)$.  
\end{enumerate}
\end{proposition}

\begin{proof}
Suppose that $G^{\Sigma }$ is almost simple and let $S=Soc(G_{\Sigma
}^{\Sigma })$. Since $G^{\Sigma }$ acts $2$-transitively on $\Sigma$, and $\left\vert\Sigma\right\vert=p^{m}+2$, by \cite[Section 2,(A)]{Ka85} one of the following holds:

\begin{enumerate}
\item $S\cong A_{p^{m}+2}$ and $p^{m}\geq 3$;

\item $S\cong PSL_{h}(u)$, with $h\geq 2$, $(h,u)\neq (2,2),(2,3)$ and $%
\frac{u^{h}-1}{u-1}=p^{m}+2$;

\item $S\cong PSU_{3}(u)$, with $u\neq 2$ and $u^{3}+1=p^{m}+2$;

\item $S\cong Sz(2^{2t+1})$, with $t\geq 1$, and $2^{4t+2}+1=p^{m}+2$;

\item $S\cong $ $^{2}G_{2}(3^{2t+1})^{\prime }$, with $t\geq 1$, and $%
3^{6t+3}+1=p^{m}+2$;

\item $S\cong Sp_{2n}(2)$, with $n\geq 3$, and $2^{2n-1}\pm 2^{n-1}=p^{m}+2$;

\item $S\cong PSL_{2}(11),M_{11}$ and $p^{m}=9$;

\item $S\cong A_{7}$ and $p^{m}=13$.
\end{enumerate}

Assume that (1) occurs. Let $\Delta_{i} \in \Sigma $ and $x_{i} \in \Delta_{i}$, then $A_{p^{m}+1}\trianglelefteq G_{\Delta
_{i}}^{\Sigma }\leq S_{p^{m}+1}$ and hence a quotient group of $G_{x _{i}}^{\Delta _{i}}$ contains $A_{p^{m}+1}$ by
Corollary \ref{glue}. On the other hand, by Theorem \ref{Monty} one of the following holds:

\begin{enumerate}

\item[(i)] $G_{x _{i}}^{\Delta _{i}} \leq \Gamma L_{2}(p^{m})$;

\item[(ii)] $G_{x _{i}}^{\Delta _{i}} \leq (Q_{8}\circ
D_{8}).S_{5}$ and $p^{m}=9$;

\item[(iii)] $G_{x _{i}}^{\Delta _{i}} \cong SL_{2}(13)$ and $%
p^{m}=27$;

\item [(iv)] $G_{x _{i}}^{\Delta _{i}} \leq (Z_{p^{m/3}-1}\times
SU_{3}(p^{m/3})).Z_{2m/3}$, with $m \equiv 0 \pmod{3}$.
\item [(v)] $G_{x _{i}}^{\Delta _{i}} \leq \Gamma Sp_{4}(p^{m/2})$, with $m$ even.
\item[(iv)] $Sz(2^{m/2}) \trianglelefteq G_{x _{i}}^{\Delta _{i}}
\leq  \left( Z_{2^{m/2}-1} \times Sz(2^{m/2}) \right).Z_{m/2}$, with $m \equiv 2 \pmod{4}$;

\item[(vii)] $G_{2}(2^{m/3})^{\prime}\trianglelefteq G_{x _{i}}^{\Delta _{i}} \leq \left( Z_{2^{m/3}-1} \times G_{2}(2^{m/3})\right).Z_{m/3}$, with $m
\equiv 0 \pmod{3}$.

\end{enumerate}

Cases (i)--(viii) bring together some of the automorphism groups of the $2$-designs listed in Theorem \ref{Monty}. For instance, $\Gamma L_{2}(p^{m})$ contains $G_{x _{i}}^{\Delta _{i}}$ when this one is as in (2.a.i) or in (2.c.i) of Theorem \ref{Monty}, $(Q_{8}\circ D_{8}).S_{5}$ contains $G_{x _{i}}^{\Delta _{i}}$ when this one is as in (2.b.iii) or in (2.c.vi) and the group in case (vii) contains the groups in (2.c.v) and for $q=2$ (2.c.vii) (the non solvable case) of Theorem \ref{Monty}.

It is easy to check that only groups in (i) for $p^{m}=2,3,4,5$ and in (vi) admit a quotient group containing $A_{p^{m}+1}$ as a normal subgroup. Actually, $p^{m}\neq 2$ by Lemma \ref{c=d}, and $p^{m}\neq 5$ since $A_{6}$ occurs only in (i) for $p^m=9$ and in (vi) for $p^m=4$. If $p^{m}=3$,
then $G$ is solvable by \cite{P} and also this case is ruled out. Thus $p^{m}=4$ and hence the assertion (I) follows from \cite{LPR}.%

Assume that (2) occurs. Thus $[u^{h-1}]:SL_{h-1}(u)\trianglelefteq
G_{\Delta_{i}}^{\Sigma }$ (e.g. see \cite[Proposition 4.1.17(II)]{KL}) and hence a quotient group of $%
G_{x_{i}}^{\Delta_{i}}$ contains $%
[u^{f(h-1)}]:SL_{h-1}(u^{f}) $ as a normal subgroup by Corollary \ref{glue}. This is clearly impossible for $h \geq 3$, hence $h=2$. Then $u^{f}=p^{m}+1$. By \cite[%
B1.1]{Rib} either $(p^{m},u^f)=(2^{3},3^{2})$, or $f=1$, $p=2$ and $u$ is a
Fermat prime, or $m=1$, $u=2$ and $p$ is a Mersenne prime. In each case $G_{\Delta_{i}}^{\Sigma}$ contains a normal Frobenius group of order $(p^{m}+1) \frac{p^{m}}{(p^{m},2)}$, with kernel of order $p^{m}+1$ and complement of order $\frac{p^{m}}{(p^{m},2)}$. Since a quotient group of $G_{x _{i}}^{\Delta _{i}}$ is isomorphic to $G_{\Delta_{i}}^{\Sigma}$ by Corollary \ref{glue}, it follows that only $G_{x _{i}}^{\Delta _{i}}$ as in (i) is admissible. Moreover, either $\mathcal{D}_{i} \cong AG_{2}(p^{m})$, where $p^{m}$ is either $8$, or $2^{2^{e}}$ with $e \geq 1$, or a Fermat prime, or $\mathcal{D}_{i}$ is as in (2.b), or $\mathcal{D}_{i}$ is as in (2.c.i) of Theorem \ref{Monty}. Assume that the former occurs. Then either $G_{x _{i}}^{\Delta _{i}} \leq \Gamma L_{1}(p^{2m})$, or $SL_{2}(p^{m})\trianglelefteq G_{x _{i}}^{\Delta _{i}}$, or $p^{m}=5$ and $SL_{2}(3)\trianglelefteq G_{x _{i}}^{\Delta _{i}}$ by \cite{F1,LiebF}. If $SL_{2}(p^{m})\trianglelefteq G_{x _{i}}^{\Delta _{i}}$, then $(p^{m}+1) \frac{p^{m}}{(p^{m},2)}$ divides $[\Gamma L_{2}(p^{m}):SL_{2}(p^{m})]$, and hence $(p^{m}-1)m$, since a quotient group of $G_{x _{i}}^{\Delta _{i}}$ contains a normal Frobenius group of order $(p^{m}+1) \frac{p^{m}}{(p^{m},2)}$. This is clearly is impossible.  Case (2.c.i) of Theorem \ref{Monty} is ruled out by the previous argument, since $SL_{2}(p^{m})\trianglelefteq G_{x _{i}}^{\Delta _{i}} \leq (Z_{q^{m/2}-1} \circ SL_{2}(p^{m})).Z_{m}$. Finally, case $p^{m}=5$ and $SL_{2}(3)\trianglelefteq G_{x _{i}}^{\Delta _{i}}$ is ruled out similarly. Thus $G_{x _{i}}^{\Delta _{i}} \leq \Gamma L_{1}(p^{2m})$ and hence $(p^{m}+1) \frac{p^{m}}{(p^{m},2)}$ divides the order of $\Gamma L_{1}(p^{2m})$. Then $\frac{p^{m}}{(p^{m},2)} \mid 2m$ and hence $p^{m}=4,16$, since $p^{m}\geq 3$ by Lemma \ref{c=d}. We reach the same conclusion if $\mathcal{D}_{i}$ is either as in (2.b) of Theorem \ref{Monty}, since $G_{x _{i}}^{\Delta _{i}} \leq \Gamma L_{1}(p^{2m})$ in this case. Hence,  $p^{m}=4,16$ in each case. If $p^{m}=4$, then $A_{5} \trianglelefteq G^{\Sigma}\leq S_{5}$ and hence the assertion (I) follows from \cite{LPR}.
 
Assume that $p^{m}=16$. Then $PSL_{2}(17) \trianglelefteq G^{\Sigma}\leq PGL_{2}(17)$. If $G^{\Sigma}\cong PGL_{2}(17)$, then  a quotient group of $ G_{x _{i}}^{\Delta _{i}}$ is isomorphic to to Frobenius group of order $17\cdot 16$. However, this is impossible since  $ G_{x _{i}}^{\Delta _{i}} \leq \Gamma L_{1}(2^{8})$. Thus $G^{\Sigma } \cong PSL_{2}(17)$, which is (III).

Cases (3), (4) and (5) yield an equation of type $u^{i}=p^{m}+1$, with $i=3,2,3$ respectively. Only (4) is admissible with $m=1$ and $u=2$ by \cite[B1.1(2)]{Rib}, however it is ruled out since $u\neq 2$ in (4). Case (6) cannot occur, since $2^{2n-1}\pm 2^{n-1}=p^{m}+2$ with $n\geq 3$ has no solutions. In (7), $G^{\Sigma}\cong PSL_{2}(11),M_{11}$ by \cite{At} and hence (II) holds. 

Finally, assume that (8) occurs. Then $G^{\Sigma}\cong A_{7}$ by \cite{At}. Moreover $\mathcal{D}%
_{i}\cong AG_{2}(13)$ and $G_{x _{i}}^{\Delta _{i}}\leq GL_{2}(13)$ by Theorem \ref{Monty}. On the other hand, a quotient
group of $G_{x _{i}}^{\Delta _{i}}$ contains $%
PSL_{2}(7)$ by Corollary \ref{glue}, since $PSL_{2}(7)\trianglelefteq
G_{\Delta _{i}}^{\Sigma }$, being $\left \vert \Sigma \right\vert =15$, and we reach a contradiction. This completes the
proof.
\end{proof}

\begin{theorem}\label{Aelo}
If $G^{\Sigma }$ is almost simple, then $p^{m}=4$ and $\mathcal{D}$ is isomorphic to one of the four $2$-$(96,20,4)$ designs constructed in \cite{LPR}.
\end{theorem}

\begin{proof}
Assume that Case (II) of Proposition \ref{Ao} occurs. Then $H$ is an irreducible subgroup of $GL_{8+t}(2)$, with $t\leq 8$, by Lemma \ref{irre}. Moreover, a quotient group of $H$ is isomorphic either to $PSL_{2}(11)$ or to $M_{11}$, since $C \leq G(\Sigma)$ by Theorem \ref{centraliz}. From \cite[Tables 8.18--8.19, 8.25--8.26, 8.35--8.36, 8.44--8.45]{BHRD}, it follows that either $H=G^{\Sigma }$
and $t=1,2$, or $H\cong SL_{2}(11)<Z_{2}.M_{12}$ for $t=2$. However, the
action of $H$ on $V_{4+t}(3)$ is irreducible only for $t=1$ by \cite{AtMod}. Therefore $%
G(\Sigma )=C$ and hence $G(\Sigma)_{x}^{\Delta}=1$ for $\Delta \in \Sigma$ and $x \in \Delta$. It follows that $G^{\Sigma}_{\Delta} \cong G_{x}^{\Delta}$ by (\ref{isom}) of Corollary \ref{glue}. Thus $G(\Sigma)_{x}^{\Delta}$ is isomorphic either to $A_{5}$ or to $P\Sigma L_{2}(9)$ according to whether $H$ is isomorphic to $PSL_{2}(11)$ or $M_{11}$
respectively by \cite{At}. On the other hand, $G_{x}^{\Delta}$ is contained in one of the groups $\Gamma L_{2}(9)$, $GSp_{4}(3)$, or $\left(D_{8}\circ Q_{8}\right) .S_{5}$ according to whether either one of (2.a.i) or (2.c.i), or (2.c.iii), or one of (2.a.iii) or (2.c.vi) occurs, respectively, by Theorem \ref{Monty}, since $p^{m}=9$ and $G_{x}^{\Delta}$ is non-solvable. However, none of these groups contains $A_{5}$ or $P\Sigma L_{2}(9)$ (see \cite[Tables 8.13--8.13]{BHRD}) and hence Case (II) is ruled out.

Assume that Case (III) occurs. Then $H$ is an irreducible subgroup of $GL_{8+t}(2)$, with $t\leq 8$, by Lemma \ref{irre}. Moreover, a quotient group of $H$ is isomorphic to $PSL_{2}(17)$, since $C \leq G(\Sigma)$ by Theorem \ref{centraliz}. If $t<8$, then $t=0$ and $H \cong PSL_{2}(17)$ by \cite[Theorem 3.1]{BP}, since $\Phi _{8}^{\ast }(2)=17$ divides the order of $H$. Thus $\left\vert G \right\vert=2^{12}\cdot 3^{2}\cdot 17$ and hence $\left\vert G_{x} \right\vert=2^{3}\cdot 17$ since $v=2^{9}\cdot 3^{2}$. Then this case is excluded since $k=2^{4}\cdot 17$ does not divide the order of $G_{x}$. Therefore $t=8$ and hence $C=V$ by Proposition \ref{qcDv}(3). It is easy to verify that $H$ is not a geometric subgroup of $GL_{16}(2)$ by using \cite[Section 4]{KL}. Thus $H$ is a nearly simple subgroup of $GL_{16}(2)$ and hence $H \cong PSL_{2}(17)$ by \cite{AtMod}.

Let $Q$ be a Sylow $17$-subgroup $G$. Simple computations with the aid of \textsf{GAP} \cite{GAP} show that $Q$ preserves a decomposition of $V=V_{1}\oplus V_{2}$, where $V_{1}$ and $V_{2}$ are the unique $Q$-invariant proper subspaces of $V$. Moreover, $V_{1}^{H}$ and $V_{2}^{H}$ are two distinct orbits each of length $18$, and $\left \vert V_{2} \cap V_{1}^{\eta} \right \vert=2^{3}$ for each $\eta \in H \setminus H_{V_{1}}$ and $\left \vert V_{1} \cap V_{2}^{\sigma} \right \vert=2^{3}$ for each $\sigma \in H \setminus H_{V_{2}}$.

The group $Q$ fixes a point $x$ of $\mathcal{D}$, since $v=2^{17}\cdot 3^{2}$. Thus $Q$ preserves the unique element $\Delta$ of $\Sigma$ containing $x$ and hence normalizes $V(\Delta)$, being $V_{x}=V(\Delta)$. Then $V_{x}$ is either $V_{1}$ or $V_{2}$, as $V_{1}$ and $V_{2}$ are the unique $Q$-invariant proper subspaces of $V$. Dualizing, $Q$ preserves a block $B$ of $\mathcal{D}$ and hence also $V_{B}$ is either $V_{1}$ or $V_{2}$. Actually, $(V_{x},V_{B})$ is either $(V_{1},V_{2})$ or $(V_{2},V_{1})$, since $[G_{B}:G_{B,x}]=2^{4}\cdot 17$ and since $G_{x}/V_{x}\cong G_{B}/V_{B}\cong F_{136}$. In particular, $x$ and $B$ are the unique point and block of $\mathcal{D}$ fixed by $Q$ respectively. Moreover, $\left \vert B \cap \Delta \right \vert =0$ and $\left \vert B \cap \Delta \right \vert= 2^{4}$ for each $\Delta^{\prime} \in \Sigma \setminus \{\Delta\}$, since $Q$ acts regularly on $\Sigma \setminus \{\Delta\}$.  

Assume that $(V_{x},V_{B})=(V_{1},V_{2})$. Then $V_{1}^{H}=\{V(\Delta^{\prime}): \Delta^{\prime} \in \Sigma \}$ and hence $\left \vert V_{B} \cap V(\Delta^{\prime}) \right \vert =2^{3}$ for each $\Delta^{\prime} \in \Sigma \setminus \{\Delta\}$. On the contrary, $V_{B}=V_{B \cap \Delta^{\prime}}$ and hence $\left \vert V_{B} \cap V(\Delta^{\prime}) \right \vert =2^{4}$ for each $ \Delta^{\prime} \in \Sigma \setminus \{\Delta\}$, since $V_{B}$ preserves each element of $\Sigma \setminus \{\Delta\}$ and since $\left \vert V_{B} \right \vert =2^{8}$ and $\left \vert B \cap \Delta^{\prime} \right \vert =2^{4}$. So, we obtain a contradiction and hence this case is excluded. The case $(V_{x},V_{B})=(V_{2},V_{1})$ is ruled out similarly, and the proof is thus completed.
\end{proof}

\section{The case where $\mathcal{D}$ is of type 2}
In this section we assume that $\mathcal{D}$ is of type 2. Hence $\mathcal{D}$ is a symmetric $2$-$((\lambda +6)\frac{\lambda ^{2}+4\lambda -1}{4}%
,\lambda \frac{\lambda +5}{2},\lambda )$ design, with $\lambda
\equiv 1,3\pmod{6}$ admitting a flag-transitive automorphism group $G$ preserving a partition $\Sigma$ of the point set of $\mathcal{D}$ in $\frac{\lambda ^{2}+4\lambda -1}{4}$ classes each of size $\lambda+6$. Then $\mathcal{%
D}_{i}$ is a $2$-$(\lambda +6,3,\lambda /\theta )$ design, with $\theta \mid
\lambda $, admitting $G_{\Delta _{i}}^{\Delta _{i}}$ as a flag-transitive, point-primitive
automorphism group for each $i=1,...,\frac{\lambda ^{2}+4\lambda -1}{4}$ by Theorem \ref{PZM}. Our aim is to prove the following result and hence completing the proof of Theorem \ref{main}.

\bigskip

\begin{theorem}
\label{t2} If $\mathcal{D}$ is of type 2, then $\mathcal{D}$ is isomorphic to the $2$-$(45,12,3)$
design constructed in \cite[Construction 4.2]{P}.
\end{theorem}

\bigskip

Our starting point is the following theorem which classifies $\mathcal{D}_{i}$. It is an application of Theorem \ref{tre} whose statement and proof are the content of the Appendix of the present paper.

\bigskip

\begin{theorem}
\label{help}Let $\mathcal{D}_{i}$ be a $2$-$(\lambda+6,3,\lambda /\theta )$-design,
with $\lambda\equiv 1,3\pmod{6}$ and $\lambda \geq 3$, admitting a
flag-transitive automorphism group $G_{\Delta _{i}}^{\Delta _{i}}$. Then $%
G_{\Delta _{i}}^{\Delta _{i}}$ acts point-$2$-transitively on $\mathcal{D}%
_{i}$, and one of the following holds:

\begin{enumerate}
\item $\mathcal{D}_{i}$ is a $2$-$\left( \frac{q^{h}-1}{q-1},3,q-1\right) $
design, $q$ is even, $\frac{q^{h}-1}{q-1}\equiv 0,1\pmod{3}$, $q-1\mid h-6$, 
$\theta =\frac{q^{h}-6q+5}{\left( q-1\right) ^{2}}$ and one of the following holds:

\begin{enumerate}
\item $PSL_{h}(q)\trianglelefteq G_{\Delta _{i}}^{\Delta _{i}}\leq P\Gamma
L_{h}(q)$.

\item $G_{\Delta _{i}}^{\Delta _{i}}\cong A_{7}$ and $(h,q)=(4,2)$.
\end{enumerate}

\item $\mathcal{D}_{i}$ is a $2$-$\left(31,3,25\right) $ design, $\theta =1$ and $PSL_{3}(5)\trianglelefteq
G_{\Delta _{i}}^{\Delta _{i}}\leq PGL_{3}(5)$.

\item $\mathcal{D}_{i}\cong AG_{h}(3)$, $h \geq 2$, $\lambda =\theta=3^{h}-6 $ and $G_{\Delta _{i}}^{\Delta _{i}}$
is of affine type.
\end{enumerate}
\end{theorem}

\bigskip

Recall that a group is called \emph{quasiprimitive} if each of its non-trivial normal subgroups is transitive. More information on quasiprimitive groups can be found in \cite{PS}.

\bigskip
\begin{lemma}
\label{ASQuasiP}If $G_{\Delta _{i}}^{\Delta _{i}}$ is almost simple then $G$
is quasiprimitive.
\end{lemma}

\begin{proof}
Suppose the contrary. Then there is a minimal normal subgroup of $G$ such
that the $G$-invariant partition $\Sigma $ is the orbit decomposition of the point set of $\mathcal{D}$ under $N$ by Lemma %
\ref{Ordine}. Then $N_{\Delta _{i}}^{\Delta _{i}}$ is a normal subgroup of $%
G_{\Delta _{i}}^{\Delta _{i}}$ acting transitively on $\Delta _{i}$ for each 
$i=1,...,\frac{\lambda ^{2}+4\lambda -1}{4}$. Then $Soc(G_{\Delta
_{i}}^{\Delta _{i}})\trianglelefteq N_{\Delta _{i}}^{\Delta
_{i}}\trianglelefteq G_{\Delta _{i}}^{\Delta _{i}}$ by \cite[Theorem
4.3B(iii)]{DM}, since $Soc(G_{\Delta _{i}}^{\Delta _{i}})$ is non-abelian simple.
Then $N_{\Delta _{i}}^{\Delta _{i}}$ acts point $2$-transitively on $%
\mathcal{D}_{i}$ by Theorem \ref{help}, and hence $\mathcal{D}_{i}^{\prime
}=\left( \Delta _{i},\left( B\cap \Delta _{i}\right) ^{N_{\Delta
_{i}}^{\Delta _{i}}}\right) $, where $B$ is any block of $\mathcal{D}$ such that $B \cap \Delta_{i} \neq \varnothing$, is
a flag-transitive, point $2$-transitive $2$-$(\lambda+6,3,\lambda ^{\prime })$
subdesign of $\mathcal{D}_{i}$, with $\lambda ^{\prime }\mid \frac{\lambda}{\theta} 
$. If $\mathcal{D}_{i}^{\prime }$ is symmetric, then $\lambda +6=7$ and $\lambda ^{\prime }=\lambda =1$, since $b=\frac{(\lambda+6)(\lambda+5)}{6}\lambda ^{\prime }$, whereas $\lambda \geq 3$ by Lemma \ref{c=d}. Thus $\mathcal{D}_{i}^{\prime }$ is non-symmetric and hence $\left\vert \left(
B\cap \Delta _{i}\right) ^{N_{\Delta _{i}}^{\Delta _{i}}}\right\vert >\lambda+6$.

Since $N_{B}\leq N_{B\cap \Delta _{i}}$, being $\Delta _{i}$ a $N$-orbit, it
follows that $\frac{\left\vert N_{B\cap \Delta _{i}}\right\vert }{\left\vert
N_{B}\right\vert }$ is an integer, and hence 
\begin{equation*}
\left\vert B^{N}\right\vert =\frac{\left\vert N\right\vert }{\left\vert
N_{B}\right\vert }=\frac{\left\vert N_{\Delta _{i}}\right\vert }{\left\vert
N_{B\cap \Delta _{i}}\right\vert }\cdot \frac{\left\vert
N_{B\cap \Delta _{i}}\right\vert }{\left\vert N_{B}\right\vert }=\frac{%
\left\vert N_{\Delta _{i}}^{\Delta _{i}}\right\vert }{\left\vert \left(
N_{\Delta _{i}}^{\Delta _{i}}\right) _{B\cap \Delta _{i}}\right\vert }\cdot 
\frac{\left\vert N_{B\cap \Delta _{i}}\right\vert }{\left\vert
N_{B}\right\vert }>(\lambda+6)\frac{\left\vert N_{B\cap \Delta _{i}}\right\vert }{%
\left\vert N_{B}\right\vert }\text{.}
\end{equation*}%
Therefore $\left\vert B^{N}\right\vert >\lambda+6$. Since $N\trianglelefteq G$ and $G
$ acts block-transitively, the block set of $\mathcal{D}$ is partitioned into $N$-orbits of
equal length greater than $\lambda+6$. Then the number of block-$N$-orbits is strictly
less than $\frac{\lambda ^{2}+4\lambda -1}{4}$, as $b=v=(\lambda+6)\cdot \frac{\lambda ^{2}+4\lambda -1}{4}$, but this contradicts \cite[Theorem 3.3]{La}, since the number of point-$N$-orbits is $\frac{\lambda ^{2}+4\lambda -1}{4}$. Thus the lemma's statement holds.
\end{proof}

\begin{proposition}
\label{dobro}$\mathcal{D}_{i}\cong AG_{h}(3)$, $h \geq 2$, $G_{\Delta _{i}}^{\Delta _{i}}
$ is of affine type and $\lambda =\theta =3^{h}-6$.
\end{proposition}

\begin{proof}
Assume that $Soc(G_{\Delta _{i}}^{\Delta _{i}})$ is non-abelian simple and
let $N$ be any minimal normal subgroup of $G$. The transitivity of $G$ on $\Sigma$ implies that $Soc(G_{\Delta _{1}}^{\Delta _{1}}) \cong Soc(G_{\Delta _{i}}^{\Delta _{i}})$ for each $i=1,...,d$, where $d=\frac{\lambda ^{2}+4\lambda -1}{4}$. Also, $Soc(G_{\Delta
_{i}}^{\Delta _{i}})\trianglelefteq N_{\Delta _{i}}^{\Delta
_{i}}\trianglelefteq G_{\Delta _{i}}^{\Delta _{i}}$ by Lemma \ref{prim}
and by \cite[ Theorem 4.3B(iii)]{DM}, since $N$ acts point-transitively on $%
\mathcal{D}$ by Lemma \ref{ASQuasiP}. Thus $N\cong Soc(G_{\Delta _{1}}^{\Delta
_{1}})^{e}$ for some $e\geq 1$, since $N$ is a minimal normal subgroup of $G$. Moreover, $Soc(G_{\Delta _{1}}^{\Delta
_{1}})\cong PSL_{h}(q)$ with $q$ even, by Theorem \ref{help}, since $d\mid \left\vert N\right\vert $ by Lemma $\ref{ASQuasiP}$. 

Let $K$ be a normal subgroup of $N$ isomorphic to $PSL_{h}(q)$. Then 
$\Sigma $ is partitioned into $K$-orbits of the equal length, say $\mu $, as $N$ is transitive on $\Sigma $. Note that $d=\frac{\lambda ^{2}-1}{4}+\lambda $ is odd, hence $K_{\Delta _{i}}$ contains a Sylow $2$-subgroup of $K$. Moreover, $K_{\Delta _{i}}\neq K$ by Lemma \ref{ASQuasiP}. Then $K_{\Delta _{i}}$ lies in a \ maximal
parabolic subgroup of $K$ by Tits'Lemma (see \cite[Theorem 1.6]{Sei}).
Thus $\qbin{h}{t}$ , with $1\leq t\leq h/2$, divides $d$ and hence 
\begin{equation}
\qbin{h}{t}\mid \left( \frac{q^{h}-1}{q-1}\right) ^{2}-8\left( 
\frac{q^{h}-1}{q-1}\right) +11.  \label{rid}
\end{equation}%
Assume that $q^{h}-1$ contains a $p$-primitive divisor $u$. Then $u$ divides 
$\qbin{h}{t}$ and hence $u=11$. Then $fh\mid 10$ by \cite[Proposition 5.2.15(ii)]{KL}, where $q=p^{f}$. Then $(q,h)=(2,10)$ or $(4,5)$ again by Theorem \ref{help}. However, none of these pairs fulfills (%
\ref{rid}). So $q^{h}-1$ does not have primitive prime divisors and hence $%
(q,h)=(2,6)$ by Zsigmondy's Theorem, since $q$ is even, $h \geq 2$ and $(q,h)\neq (2,2)$. However, $(q,h)=(2,6)$ does not
fulfill (\ref{rid}). Thus $\mathcal{D}_{i}\cong AG_{h}(3)$, $h \geq 2$, $G_{\Delta
_{i}}^{\Delta _{i}}$ is of affine type and $\lambda =\theta =3^{h}-6$ by
Theorem \ref{help}.
\end{proof}

\bigskip

\begin{proof}[Proof of Theorem \ref{t2}]
$\mathcal{D}_{i}\cong AG_{h}(3)$, $h \geq 2$, $G_{\Delta _{i}}^{\Delta _{i}}$ is of
affine type and $\lambda =\theta =3^{h}-6$ by Proposition \ref{dobro}. Let $%
x_{i}\in \Delta _{i}$ and $x_{j}\in \Delta _{j}$, with $i\neq j$. Then $%
x_{i}\neq x_{j}$ and hence there are $\lambda $ distinct blocks, say $%
B_{1},...,B_{\lambda }$, incident with them. For each $t=1,...,\lambda $,
let $B_{t}^{(s)}$, where $s=1,...,\lambda $, be the distinct
blocks of $\mathcal{D}$ overlapping with $B_{t}$ on $\Delta _{i}$ including 
$B_{t}$, since $\theta =\lambda $. It is clear that $B_{t_{1}}^{(s_{1})}=B_{t_{2}}^{(s_{2})}$ if, and only if, $t_{1}=t_{2}$ and $s_{1}=s_{2}$. Then $B_{t}^{(s)}$, where $t,s=1,...,\lambda$, are $\lambda^{2}$ distinct blocks of $\mathcal{D}$ incident with $x_{i}$. Thus $\lambda ^{2}\leq r$ and hence $%
3\leq \lambda \leq \allowbreak \frac{\lambda +5}{2}$. Therefore, $\lambda
=\theta =3$ and $h=2$, as $\lambda =\theta =3^{h}-6$ and $h \geq 2$. The assertion now
follows from \cite[Corollary 1.2]{P}.
\end{proof}

\bigskip

Now, Theorem \ref{main} follows from Theorems \ref{t1} and \ref{t2}. 

\section{Appendix}

The aim of this section is to prove the following classification theorem for
flag-transitive $2$-$(v,3,\lambda )$ designs with $v\equiv 1,3\pmod{6}$ and $%
\lambda \mid v-6$. An application of this result is Theorem \ref{help}, which is central in classifying symmetric $2$-designs of type 2.

\bigskip

\begin{theorem}
\label{tre}Let $\mathcal{D}$ be a $2$-$(v,3,\lambda )$-design, with $v\equiv
1,3\pmod{6}$ and $\lambda \mid v-6$, admitting a flag-transitive
automorphism group $G$. Then $G$ acts point-$2$-transitively on $\mathcal{D}$%
, and one of the following holds:

\begin{enumerate}
\item $\mathcal{D}$ is a $2$-$\left( \frac{q^{h}-1}{q-1},3,q-1\right) $
design, $q$ even, $\frac{q^{h}-1}{q-1}\equiv 0,1\pmod{3}$, $q-1\mid h-6$, and one of he following holds:

\begin{enumerate}
\item $PSL_{h}(q)\trianglelefteq G\leq P\Gamma L_{h}(q)$.

\item $G\cong A_{7}$ and $(h,q)=(4,2)$.
\end{enumerate}

\item $\mathcal{D}$ is a $2$-$\left( 31,3,25 \right) $ design and $PSL_{3}(5)\trianglelefteq G\leq
PGL_{3}(5)$.

\item $\mathcal{D}\cong AG_{h}(3)$, $h \geq 2$, and $G$ is of affine type.
\end{enumerate}
\end{theorem}

\bigskip 

As it can been deduced from the proof following propositions, the $2$-designs in (1) and (2) do exist. Indeed, their point set is that of $PG_{h-1}(q)$, where $h=3$ in (2), and their block sets consist of all $3$-subsets of
collinear points and non-collinear points of $PG_{h-1}(q)$ respectively. 

\bigskip 

\begin{lemma}
\label{PP}If $\mathcal{D}$ is a $2$-$(v,3,\lambda )$ design, with $v\equiv 1,3\pmod{6}$ and $%
\lambda \mid v-6$, admitting a flag-transitive automorphism group $G$, then one
the following holds:

\begin{enumerate}
\item $G$ acts point-$2$-transitively on $\mathcal{D}$.

\item $G$ is a point-primitive rank $3$ automorphism group of $\mathcal{D}$. For any point $x$ of $\mathcal{D}$ the set $\mathcal{P}- \left\lbrace
x \right\rbrace$ is split into two $G_{x}$-orbits each of length $(v-1)/2$. Moreover $r/\lambda=(v-1)/2$. 

\end{enumerate}
\end{lemma}

\begin{proof}
Since $(v-1,k-1)\leq 2$, either (1) or the first part of (2) holds by \cite[Corollary 4.6]{Ka69}. Hence, it remains to prove $r/\lambda=(v-1)/2$.\\
Assume that the first part of (2) occurs. Let $y_{j}^{G_{x}}$, $j=1,2$, be the two $G_{x}$-orbits partitioning $\mathcal{P}- \left\lbrace
x \right\rbrace$ and let $B$ any block of $\mathcal{D}$ incident with $x$. Then $\left\vert B\cap y_{j}^{G_{x}}\right\vert =1$ for each $j=1,2$, since $G$ is flag-transitive and $k=3$. On the other hand, $(y_{j}^{G_{x}},B^{G_{x}})$ is a tactical
configuration for each $j=1,2$ by \cite[1.2.6]{Demb}. Thus $\left\vert
y_{j}^{G_{x}}\right\vert \lambda =r\left\vert B\cap y_{j}^{G_{x}}\right\vert $ and hence $r/\lambda=(v-1)/2$, since $\left\vert y_{j}^{G_{x}}\right\vert=(v-1)/2$ and $\left\vert B\cap y_{j}^{G_{x}}\right\vert =1$ for each $j=1,2$.
\end{proof}

\bigskip

\begin{proposition}
\label{tre2tr}$G$ acts point-$2$-transitively on $\mathcal{D}$.
\end{proposition}

\begin{proof}
Assume that (2) of Lemma \ref{PP} holds. By the O'nan-Scott Theorem (e.g. see \cite{DM}), one of the following holds:

\begin{enumerate}
\item $S\times S\trianglelefteq G\leq S\wr Z_{2}$, where $Soc(S)$ is a $2$%
-transitive non-abelian simple group of degree $n_{0}$, $n_{0}\geq 5$, and $v=n_{0}^{2}$.

\item $Soc(G)$ is non-abelian simple.

\item $Soc(G)$ is an elementary abelian $p$-group for some prime $p$.
\end{enumerate}

The groups in (1) are classified in \cite{C}, those in (2) are determined
in \cite{Ba} when the socle is alternating, in \cite{KaLibl3} when the socle
is classical, in \cite{LS} when the socle is an exceptional group of Lie
type or a sporadic group. Finally, the groups in (3) are classified in \cite%
{F, FK, Lieb2}.

In (1) the $G_{x}$-orbits on $\mathcal{P}-\left\{ x\right\} $ have length $%
2(n_{0}-1)$ and $(n_{0}-1)^{2}$, and these are distinct as $n_{0}\geq 5$, so
this case cannot occur. It is straightforward to check that $%
A_{7}\trianglelefteq G\leq S_{7}$ and $v =21$ is the unique admissible
case arising from (2) (see \cite{Ba,KaLibl3,LS}. An overview of
the subdegrees for $G$ is provided in \cite{Dev}). Then $\mathcal{D}$ is a $2
$-$(21,3,\lambda )$, with $\lambda \mid 15$. Actually $\lambda =3$, since $%
r=10\lambda $ must divide the order of $G$, and since $\lambda =1$ is ruled
out in \cite{Del}. By \cite{At}, if $x$ is any point of $\mathcal{D}$, either $G_{x}\cong S_{5}$ or $G_{x}\cong S_{5}\times Z_{2}$ according to
whether $G\cong A_{7}$ or $S_{7}$ respectively. Hence either $\left\vert
G_{x,B}\right\vert =4$ or $\left\vert G_{x,B}\right\vert =8$, respectively, where $B$ is
any block of $\mathcal{D}$ incident with $x$. Moreover, $G_{x,B}=G(B)$,
since $B$ contains exactly one point from each $G_{x}$-orbit. However, this
is impossible, since $G_{x,y}\cong S_{3}$ or $S_{3}\times Z_{2}$ for any
point of $y$ distinct from $x$ according to
whether $G\cong A_{7}$ or $S_{7}$ respectively.

Finally, assume that $G$ is as in (3). Then the admissible groups listed below
arise by \cite[Tables 12--14]{Lieb2}, by \cite[Theorem 1.1 and
subsequent Remark]{F} and by \cite[Corollary 1.3 and Theorems 3.10 and 5.3]{FK}:

\begin{enumerate}
\item[(i).] $G_{0}\leq \Gamma L_{1}(p^{d})$;

\item[(ii).] $G_{0}\leq N_{GL_{2}(7)}(Q_{8})$ and the non-trivial $G_{0}$-orbits
have length $24$;

\item[(iii).] $G_{0}\leq N_{GL_{2}(23)}(Q_{8})$ and the non-trivial $G_{0}$-orbits
have length $264$;

\item[(iv).] $G_{0}\leq N_{GL_{2}(47)}(Q_{8})$ and the non-trivial $G_{0}$-orbits
have length $1104$;

\item[(v).] $SL_{2}(5)\trianglelefteq G_{0}$ acting on $V_{4}(3)$ and the non-trivial $G_{0}$-orbits have length $40$.
\end{enumerate}

Clearly, $-1\in G_{0}$ in (v). Also in (ii)--(iv) $-1\in G_{0}$, since $4\mid
\left\vert G_{0}\cap SL_{h}(q)\right\vert $ for $(h,q)=(2,7)$ or $(2,23)$,
and since the Sylow $2$-subgroups of $SL_{2}(q)$ for $q$ odd are generalized quaternion groups with $-1$
as their unique involution. Finally, if $G_{0}$ is as in (i), then $%
G_{0}=\left\langle \bar{\omega}^{t},\bar{\omega}^{e}\bar{\alpha} ^{s}\right\rangle 
$, where $\omega$ is a primitive element of $GF(q)$, $\bar{\omega}:x\rightarrow\omega x$, $\bar{\alpha}:x\rightarrow x^{p}$, and where $t,e,s$ satisfy the constraints given in \cite[Theorem 3.10]{FK}.
Moreover $p$ is odd, as $v=p^{d}$ and $(v-1)/2$ is an integer. Arguing as in \cite[Lemma 63]{BFM}, we
see that $-1\in G_{0}$. Thus $-1\in G_{0}$ in each case. \\
Let $x\in \mathcal{D}$, $x\neq 0$, then $-1$ preserves the $\lambda $
blocks of $\mathcal{D}$ incident with $\pm x$. Moreover, $-1$ preserves at
least one of them, as $\lambda $ is odd. Therefore, $B=\left\{ 0,x,-x\right\} $ and hence $\left\vert G_{0,B}\right\vert =2\left\vert
G_{0,x}\right\vert $. Thus $r=\left[ G_{0}:G_{0,B}\right] =\frac{\left\vert
G_{0}\right\vert }{2\left\vert G_{0,x}\right\vert }=\frac{p^{d}-1}{4}$, with $p^{d} \equiv 1 \pmod{4}$,
whereas $r=\frac{p^{d}-1}{2}\lambda $. Thus (i)--(v) are excluded, and hence $G$ acts point-$2$-transitively on $\mathcal{D}$ by Lemma \ref{PP}. 
\end{proof}

\bigskip

\begin{theorem}
\label{treAS}If $G$ is almost simple, then one of the following holds:

\begin{enumerate}
\item $\mathcal{D}$ is a $2$-$\left( \frac{q^{h}-1}{q-1},3,q-1\right) $
design, $q$ is even, $\frac{q^{h}-1}{q-1}\equiv 0,1\pmod{3}$, $q-1\mid h-6$,
and one of he following holds:

\begin{enumerate}
\item $PSL_{h}(q)\trianglelefteq G\leq P\Gamma L_{h}(q)$;

\item $G\cong A_{7}$ and $(h,q)=(4,2)$.
\end{enumerate}

\item $\mathcal{D}$ is a $2$-$\left( 31,3,25\right) $ design and $PSL_{3}(5)%
\trianglelefteq G\leq PGL_{3}(5)$.
\end{enumerate}
\end{theorem}

\begin{proof}
Assume that $G$ is an almost simple, point-$2$-transitive automorphism group
of $\mathcal{D}$. Then $G$ is listed in \cite[Section 2, (A)]{Ka85}. Moreover, $v\equiv 1,3%
\pmod{6}$. Thus, one of the following holds:

\begin{enumerate}
\item[(i).] $Soc(G)\cong A_{v}$ and $\lambda \mid v-6$, $v\equiv 1,3\pmod{6}$;

\item[(ii).] $Soc(G)\cong PSL_{h}(q)$, $h \geq 2$, $\frac{%
q^{h}-1}{q-1}\equiv 1,3\pmod{6}$ and $(h,q)\neq (2,2)$, and $\lambda \mid \frac{q^{h}-1}{q-1}-6$;

\item[(iii).] $Soc(G)\cong PSU_{3}(2^{2m+1})$, $m>0$, and $\lambda \mid 2^{6m+3}-5$;

\item[(iv).] $Soc(G)\cong A_{7}$ and $\lambda  \mid 9$.
\end{enumerate}

If $G$ acts point-$3$-transitively on $\mathcal{D}$, then $v-2= \lambda \leq v-6$, and we reach a contradiction. Thus $G$ cannot act point-$3$-transitively on $\mathcal{D}$ and hence (i) is ruled out.

Assume that (ii) holds. If $h=2$, then the point-$2$-transitive actions of $G$ on $\mathcal{D}$ and on $PG_{1}(q)$ are equivalent. Hence, we may identify the point set of $\mathcal{D}$ with that of $PG_{1}(q)$. Also $q$ is even, since $v=q+1$ and $v \equiv 1,3 \pmod{6}$. Then $G$ acts point-$3$-transitively on $\mathcal{D}$, which is not the case. Thus $h>2$ and hence $Soc(G)$ has two $2$-transitive
permutation representations of degree $\frac{q^{h}-1}{q-1}$. These are the one on the set of points of $PG_{h-1}(q)$, and the other on the set of hyperplanes of $PG_{h-1}(q)$. The two conjugacy classes in $PSL_{h}(q)$ (resp. in $P \Gamma L_{h}(q)$) of the point-stabilizers and hyperplane-stabilizers are fused by a polarity of $PG_{h-1}(q)$. Thus, we may identify the point set of $\mathcal{D}$ with that of $PG_{h-1}(q)$.

Assume that $B=\left\{ x,y,z\right\} $ is contained in a line $\ell $ of $%
PG_{h-1}(q)$. Then each block of $\mathcal{D}$ is contained in a unique line
of $PG_{h-1}(q)$, being $G$ transitive one the block set of $\mathcal{D}$
as well as on the line set of $PG_{h-1}(q)$. Moreover, $\lambda =q-1$, since\ $%
G_{\ell }^{\ell }\cong PGL_{2}(q)$ acts $3$-transitively on $\ell $. Since $%
\lambda \mid v-6$ and since $v=\frac{q^{h}-1}{q-1}$, it follows that 
\begin{equation*}
\left( q-1\right) \mid \sum_{i=0}^{h-1}(q^{i}-1)+h-6\text{.}
\end{equation*}%
Therefore, $q-1\mid h-6$. If $q$ is odd then $h$ is even, whereas $\frac{%
q^{h}-1}{q-1}$ is odd. Thus $q$ is even and (1a) follows.

Assume that $B=\left\{ x,y,z\right\} $ consists of non-collinear points of $PG_{h-1}(q)$. Then there is a plane containing $B$, say $\pi $. Let $\ell ^{\prime
}$ be the line of $PG_{h-1}(q)$ containing $x$ and $y$, and let $\mathcal{F}$ be the plane pencil with axis $\ell ^{\prime }$. Then the blocks of $\mathcal{D}$
containing $x,y$ are triangles of $PG_{h-1}(q)$ with the third vertex contained in a (unique) plane
of $\mathcal{F}$. By \cite[Proposition 4.1.17(II)]{KL}, $\left[
q^{2(h-2)}\right] :SL_{h-2}(q)\leq G_{x,y}\leq G_{\ell ^{\prime }}$, where $%
\left[ q^{2(h-2)}\right] $ fixes $\ell ^{\prime }$ pointwise and induces the
translation group on the affine plane $\pi^{\ell ^{\prime }}$ for each $\pi \in \mathcal{F}$, and $SL_{h-2}(q)$ fixes $%
\ell ^{\prime }$ pointwise and permutes transitively the elements of $%
\mathcal{F}$. Therefore $G_{x,y}$ acts
transitively one the set of all triangles of $PG_{h-1}(q)$ having $x$, $y$ as two of the
three vertices, hence $$\lambda =q^{2}\qbin {h-2}{3-2}=q^{2}\frac{%
q^{h-2}-1}{q-1}.$$ It follows from $\lambda \mid v-6$ that%
\begin{equation*}
q^{2}\frac{q^{h-2}-1}{q-1}\mid \frac{q^{h}-1}{q-1}-6
\end{equation*}%
and hence $h=3$ and $q=5$. Thus we obtain (2).

Assume that (iii) holds. Then the action of $G$ on the
point set of $\mathcal{D}$ and that on the point set of the Hermitian
unital $\mathcal{H}(2^{2m+1})$ are equivalent. Thus we may identify the point set of $\mathcal{D}$ with that of $\mathcal{H}(2^{2m+1})$.\\ 
 If $B=\left\{ x,y,z\right\} $
is contained in a line $\ell ^{\prime }$ of $\mathcal{H}(2^{2m+1})$, then
each block of $\mathcal{D}$ is contained in a unique line of $\mathcal{H}%
(2^{2m+1})$, being $G$ transitive on the block set of $\mathcal{D}$ as well
as on the line set of $\mathcal{H}(2^{2m+1})$. Thus $\lambda =2^{2m+1}-1$, since $Soc(G)_{\ell ^{\prime }}^{\ell ^{\prime }}\cong
PGL_{2}(q)$ acts $3$-transitively on $\ell ^{\prime }$. So $
2^{2m+1}-1 \mid 2^{6m+3}-5$, as $\lambda \mid v-6$, which is a contradiction.

Assume that $B=\left\{ x,y,z\right\} $ consists of points of $\mathcal{H}(2^{2m+1})$ in a triangular
configuration. Note that $Soc(G)_{x}=Q:C$, where $Q$ is a Sylow $2$-subgroup of $Soc(G)$ acting regularly on $\mathcal{H}(2^{2m+1})\setminus \{x\}$, and $C$ is a cyclic group of order $2^{2(2m+1)}-1$. If $a$ is any line of $\mathcal{H}(2^{2m+1})$ incident with $x$, then $Z(Q)$ preserves $a$ and acts regularly on $a \setminus \{x\}$. Also, $Q/Z(Q):C$ is a Frobenius group of order $2^{2(2m+1)}(2^{2(2m+1)}-1)$ acting $2$-transitively on the set of lines of $\mathcal{H}(2^{2m+1})$ incident with $x$ (e.g. see \cite[Satz II.10.12]{Hup}). Thus $C=Soc(G)_{x,y}$ acts semiregularly on $\mathcal{H}(2^{2m+1})\setminus s$, where $s$ is the line incident with $x,y$, and hence each $G_{x,y}$-orbit in $\mathcal{H}(2^{2m+1})\setminus s$ is of length a multiple $2^{4m+2}-1$. It follows that the number of blocks of $\mathcal{D}$ incident with $x,y$ is\
a multiple of $2^{4m+2}-1$. Therefore $2^{4m+2}-1 \mid \lambda$, and hence $2^{4m+2}-1 \mid 2^{6m+3}-5$, as $\lambda \mid v-6$ and $v=2^{6m+3}+1$, which is a contradiction. So, (iii) cannot occur.   

In (iv) the action of $Soc(G)\cong A_{7}$ on the point set of $\mathcal{D%
}$ is equivalent to one of the two $2$-transitive permutation representations of degree $15$, namely the ones on the set of points and on the set of planes of $PG_{3}(2)$ respectively. Arguing as in (1), we may identify the points $\mathcal{D}$ with the points of $PG_{3}(2)$. By \cite{AtMod}, $%
G\cong A_{7}$, $G_{x}\cong SL_{3}(2)$ and $%
G_{x,y}\cong A_{4}$. Thus $T\leq G(B)$, where $T\cong E_{4}$ and $B$ is a
block of $\mathcal{D}$ incident with the points $x,y$, since $\lambda \mid 9$. Since $%
N_{G_{x}}(T)\cong S_{4}$ and $G_{x,y}\cong A_{4}$, it follows that $T$ fixes exactly two points on $%
PG_{3}(2) \setminus\{x\}$ and these are necessarily collinear with $x$ (clearly $y$ is one of these). Thus $B$ is a line of $PG_{3}(2)$ and hence $\mathcal{D}\cong
PG_{3}(2)$, which is (1b).
\end{proof}

\bigskip

\begin{theorem}
\label{treAff}If $G$ is of affine type, then $\mathcal{D}\cong AG_{h}(3)$.
\end{theorem}

\begin{proof}
Assume that $Soc(G)$ is an elementary abelian $p$-group for some prime $p$ and denote it by $T$. Since $G$ acts point-$2$-transitively on $\mathcal{D}$, we may identify the point set of $\mathcal{D}$ with a $h$-dimensional $GF(p)$-space $V$ in a way that $T$ is the translation group of $V$ and $G=T:G_{0}$, with $G_{0}$ acting irreducibly on $V$.\\For each divisor $n$ of $h$ the group $\Gamma L_{n}(p^{h/n})$ has a natural irreducible action on $V$. As $G_{0}$ acts irreducibly on $V$, we may choose $n$
to be minimal such that $G_{0}\leq \Gamma L_{n}(p^{h/n})$ in this action and
write $q=p^{h/n}$. Also $q$ is odd, since $v=q^{n}-6$ and $v \equiv 1,3 \pmod{6}$. \\
Let $B$ be any block of $\mathcal{D}$ incident with $0$. Assume that $B$ is not contained in any $1$-dimensional $GF(p)$-subspace of 
$V$. Suppose that $-1\in G_{0}$. If $x$ is a point of $ \mathcal{D}$ and $x\neq 0$, it
follows that $-1$ preserves the $\lambda $ blocks incident with $\pm x$.
Actually, $-1$ preserves at least one of them since $\lambda $ is odd. Thus $%
B=\left\{ 0,x,-x\right\} \subseteq \left\langle x\right\rangle _{GF(p)}$, as $k=3$,
but this contradicts our assumption. Therefore $-1\notin G_{0}$, and hence $%
SL_{n}(q)\trianglelefteq G_{0}$, with $n$ odd and $n \geq 3$, by \cite[Section 2, (B)]{Ka85} and by \cite[Lemma 3.12]{BF}. 

Assume that $B\subseteq \left\langle u\right\rangle _{GF(q)}$ for some non-zero vector $u$ of $V$. Then any block of $\mathcal{D}$ is
contained in a unique line of $AG_{n}(q)$, since $G$ acts block-transitively on $\mathcal{D}$ and transitively on the line set of $AG_{n}(q)$. In particular the $\lambda$ blocks of $\mathcal{D}$ incident with $\pm u$ are contained in $\left\langle u\right\rangle _{GF(q)}$. Since $%
AGL_{1}(q) \trianglelefteq G_{\left\langle u\right\rangle _{GF(q)}}^{\left\langle u\right\rangle _{GF(q)}}$ and $q$ is odd, there is a $2$%
-element $\phi$ in $G_{0,\left\langle u\right\rangle _{GF(q)}}$ inducing $-1$
on $\left\langle u\right\rangle _{GF(q)}$. Therefore, $\phi $ preserves at
least one of the $\lambda $ blocks incident with $\pm u$, say $B^{\prime}$, since $\lambda$ is odd. Hence $B^{\prime}=\left\{ 0,u,-u\right\} $, as $B^{\prime}\subseteq \left\langle u\right\rangle _{GF(q)}$. So, $B$ is contained in a $1$-dimensional $GF(p)$-subspace of $V$, since $G$ acts-flag-transitively on $\mathcal{D}$, which is contrary to our assumption.\\
Assume that $B$ is not contained in any $1$-dimensional $GF(q)$-subspace of $V$. Nevertheless, $B\subset \pi $ where $\pi$ is a $2$-dimensional $GF(q)$-subspace of $V$. Thus $%
r=\qbin{n}{2}\mu$, where $\mu$ denotes the number of blocks of $\mathcal{D}$ contained in any $2$-dimensional $GF(q)$-subspace of $V$, since $G$ acts flag-transitively on $\mathcal{D}$ and since $SL_{n}(q)\trianglelefteq G_{0}$ acts
transitively on the set of $2$-dimensional $GF(q)$-subspaces of $V$. Therefore 
\begin{equation} \label{miumiu}
\qbin{n}{2}\mu=\frac{%
q^{n}-1}{2}\lambda, 
\end{equation}
and hence 
\begin{equation*}
\frac{\left( q^{n}-1\right) \left( q^{n-1}-1\right) }{\left( q^{2}-1\right)
\left( q-1\right) }\mid \frac{\left( q^{n}-1\right) \left( q^{n}-6\right) }{2%
},
\end{equation*}%
since $\lambda \mid q^{n}-6$. Thus $2\left( q^{n-1}-1\right) \mid \left( q^{2}-1\right) \left(
q-1\right)(q-6)$ and hence $n = 3$, since $n$ is odd and $n \geq 3$. Now substituting $n=3$ in (\ref{miumiu}), it results $\lambda=\frac{2\mu}{q-1}$. So, $\mu$ is odd as $\lambda \mid q^{n}-6$ and $q$ is odd.\\
Lastly, $SL_{3}(q)$ contains an involution $\sigma$ inducing $-1$ on $\pi$. Then $\sigma$ preserves at least one of the $\mu$ blocks of $\mathcal{D}$ contained in $\pi$, say $B^{\prime\prime}$, since $\mu$ is odd. Then $B^{\prime\prime}=\{0,y,-y \}$ for some non-zero $y$ in $\pi$, since $k=3$ and $\sigma$ induces $-1$ on $\pi$. Hence $B=\{0,y^{\beta},-y^{\beta}\}$ for some $\beta \in G_{0}$, since $G$ is flag-transitive on $\mathcal{D}$, but this contradicts our assumption.\\
Assume that $B\subseteq \left\langle u\right\rangle _{GF(p)}$ for some non-zero
vector $u$ of $V$. Then $G_{B} \leq G_{\left\langle u\right\rangle}$ and hence $S_{3} \cong (G_{B})^{\left\langle u\right\rangle}\leq G_{\left\langle u\right\rangle }^{\left\langle u\right\rangle }\cong AGL_{1}(p)$, since $G$ acts $2$-transitively on $V$. Thus $p=3$, $B=\left\langle u\right\rangle _{GF(3)}$ and hence the assertion follows.
\end{proof}

\bigskip

Now, Theorem \ref{tre} follows from Proposition \ref{tre2tr} and
Theorems \ref{treAS} and \ref{treAff}. Finally, Theorem \ref{t2} follows from Theorem \ref{tre}.

\end{document}